\documentclass{article}
\usepackage{arxiv}
\usepackage{mathtools}
\usepackage{mathrsfs}
\usepackage{todonotes}
\usepackage{amsthm}
\usepackage{amsmath, amssymb}
\usepackage[inkscapelatex=false]{svg}
\usepackage{graphicx} 
\usepackage{multirow}
\usepackage{caption}
\usepackage{subcaption}
\usepackage{cleveref}
\usepackage[ddmmyyyy]{datetime}
\usepackage{url}
\usepackage{authblk}

\newcommand{\F}{\mathcal{F}}
\newcommand{\E}{\mathcal{E}}
\newcommand{\V}{\mathcal{V}}
\newcommand{\N}{\mathcal{N}}
\newcommand{\HH}{\mathcal{H}}

\newcommand{\R}{\mathbb{R}}

\newtheorem{definition}{Definition}
\newtheorem{lemma}{Lemma}
\newtheorem{theorem}{Theorem}
\newtheorem{remark}{Remark}

\title{A BDDC Preconditioner for the Cardiac EMI Model in Three Dimensions}
\author[1,*]{Fritz Göbel}
\author[2]{Ngoc Mai Monica Huynh}
\author[3]{Fatemeh Chegini}
\author[2]{Luca Franco Pavarino}
\author[3]{Martin Weiser}
\author[4]{Simone Scacchi}
\author[1]{Hartwig Anzt}
\affil[1]{School of Computation, Information and Technology, Technical University of Munich, Bildungscampus 2, 74076 Heilbronn, Germany.}
\affil[2]{Dipartimento di Matematica, Universita` degli Studi di Pavia, Via Ferrata, 27100 Pavia, Italy.}
\affil[3]{Zuse Institute Berlin, Takustraße 7, 14195 Berlin, Germany.}
\affil[4]{Dipartimento di Matematica, Universita` degli Studi di Milano, Via Saldini 50, 20133 Milano, Italy.}
\affil[*]{Corresponding author: Fritz Göbel, fritz.goebel@tum.de}
\date{\today}

\begin{document}

\maketitle

\begin{abstract}
We analyze a Balancing Domain Decomposition by Constraints (BDDC) preconditioner for the solution of three dimensional composite Discontinuous Galerkin discretizations of reaction-diffusion systems of ordinary and partial differential equations arising in cardiac cell-by-cell models like the Extracellular space, Membrane and Intracellular space (EMI) Model. These microscopic models are essential for the understanding of events in aging and structurally diseased hearts which macroscopic models relying on homogenized descriptions of the cardiac tissue, like Monodomain and Bidomain models, fail to adequately represent. The modeling of each individual cardiac cell results in discontinuous global solutions across cell boundaries, requiring the careful construction of dual and primal spaces for the BDDC preconditioner. We provide a scalable condition number bound for the precondition operator and validate the theoretical results with extensive numerical experiments.
\end{abstract}

\section{Introduction}

This work aims at generalizing a previous convergence analysis of BDDC preconditioners for cardiac cell-by-cell models \cite{2D-proof} to three spatial dimensions. Such models on the microscopic level have been studied over the past years in order to understand events in aging and structurally diseased hearts which macroscopic Monodomain and Bidomain models \cite{potse2006comparison, bidomain-1, bidomain-2} that are based on a homogenized description of the cardiac tissue fail to adequately represent. In such events, reduced electrical coupling leads to large differences in behavior between neighboring cells, which requires careful modeling of each individual cardiac cell.

We consider the EMI (Extracellular space, cell Membrane and Intracellular space) model, which has been introduced and analyzed in several works, such as \cite{potse2017cinc, EMI-3, EMI-2, EMI-1, EMI-4, EMI-5} and has been at the core of the EuroHPC project MICROCARD (website: https://cordis.europa.eu/project/id/955495). Throughout the course of this project, Balancing Domain Decomposition by Constraints (BDDC) preconditioners (we refer to \cite{pechstein2017bddc, dd-book} for extensive explanations of such algorithms) have been identified to be an efficient choice for composite Discontinuous Galerkin (DG) type discretizations of cardiac cell-by-cell models and a theoretical analysis of this method has been presented in \cite{2D-proof}, while a preliminary study on its integration with time-stepping methods can be found in \cite{chegini2023coupled}. Here, a key step was the careful construction of extended dual and primal spaces for the degrees of freedom, allowing for continuous mapping in the BDDC splitting while still honoring the discontinuities across cell boundaries as they occur in the EMI model. Up to now, the theoretical analysis in three dimensions remains an open question, as the introduction of edge terms into the primal space requires additional constraints. In this work, we leverage results from \cite{sarkis-3D} as well as standard sub-structuring theory arguments from \cite{dd-book} in order to close this gap. Other paths for the solution of EMI models have been investigated, ranging from boundary element methods \cite{bem2024} as discretization choice, to multigrd solvers \cite{budivsa2024algebraic}, overlapping Schwarz preconditioners \cite{huynh2025gdsw} and other iterative solvers built upon ad-hoc spectral analysis \cite{benedusi2024modeling}. Other approaches to domain decomposition preconditioning for DG type problems have been studied e.g. in \cite{antonietti2007esaim} and \cite{ayuso2014multilevel}.

For the sake of completeness, we will first give an overview of a simplified EMI model and its time and space discretizations in \Cref{sec:emi}, referring to \cite{EMI-1, EMI-4, EMI-5} for details on the EMI model and to \cite{2D-proof} for the derivation of the time and space composite-DG discretizations. We introduce the finite element spaces, describe the BDDC preconditioner for this application, and recall the Lemmas from literature necessary to prove its convergence in \Cref{spaces} before providing a convergence proof, the main contribution of this paper, in \Cref{sec:proof}. Finally, \Cref{results} presents a numerical study supporting the derived theory with practical results.

\section{The EMI model} \label{sec:emi}
As in \cite{2D-proof}, we consider $N$ myocytes immersed in extracellular liquid, together forming the cardiac tissue $\Omega$ while focusing on the case that $\Omega \subset \R^3$. We assign the extracellular subdomain to be $\Omega_0$ and each of the $N$ cardiac cells to be separate subdomains $\Omega_1,\dots,\Omega_N$, interacting with the surrounding extracellular space through ionic currents and with their neighboring myocytes via gap junctions, which are special protein channels allowing for the passage of ions directly between two cells \cite{ionic}. We assume a partition of the cardiac tissue $\Omega$ into $N + 1$ non-overlapping subdomains $\Omega_i$ such that $\overline \Omega = \cup_{i = 0}^N \overline \Omega_i$, $\Omega_i \cap \Omega_j = \emptyset$ if $i \neq j$.
A simplified geometry with two idealized cells is represented in Figure \ref{fig:two_cells}.

\begin{figure}
    \centering
    \includegraphics[width=0.5\textwidth]{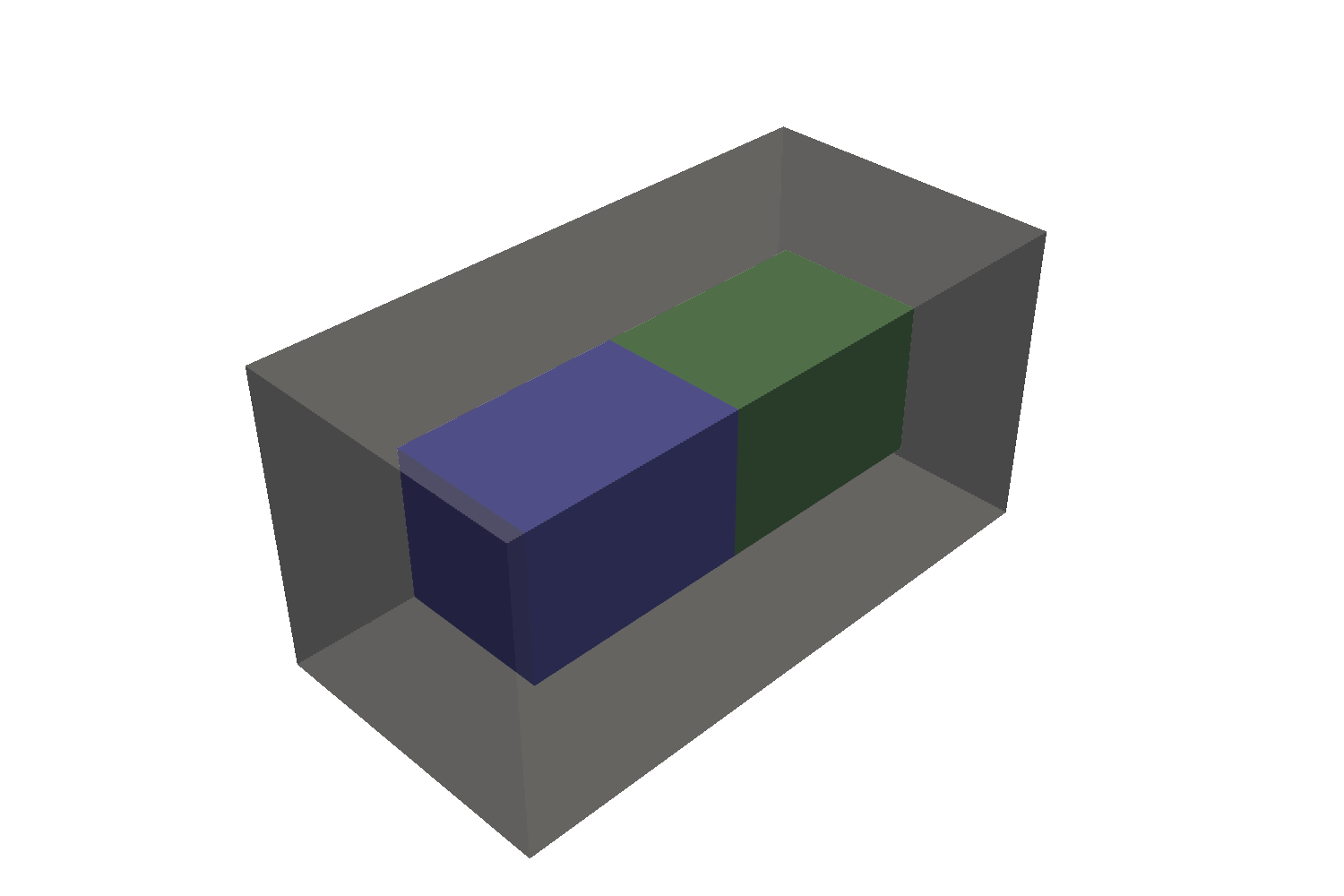}
    \caption{Visualization of two cells $\Omega_1$ (green) and $\Omega_2$ (blue) floating in extracellular liquid $\Omega_0$ (grey). For this three-dimensional example, we have to consider interface terms for ionic currents between extracellular space and the cells over $F_{01} = \partial\Omega_0 \cap \partial\Omega_1$ and $F_{02} = \partial \Omega_0 \cap \partial \Omega_2$ and ionic currents through gap junctions between the cells over $F_{12} = \partial\Omega_1\cap\partial\Omega_2$. For the BDDC preconditioner, it is important to also consider edge terms on $E_{0, \{1, 2\}} = \partial\Omega_0\cap\partial\Omega_1\cap\partial\Omega_2$.}
    \label{fig:two_cells}
\end{figure}

The EMI model is described by the equations
\begin{equation}
\label{eq:emi}
    \begin{cases}
        -\text{div}(\sigma_i \nabla u_i) = 0 & \text{in } \Omega_i, i = 0,\dots,N \\
        -n_i^T\sigma_i\nabla u_i = C_m\frac{\partial v_{ij}}{\partial t} + F(v_{ij}, c, w) & \text{on } F_{ij} = \partial\Omega_i \cap \partial\Omega_j, i \neq j \\
        n^T\sigma_i\nabla u_i = 0 & \text{on } \partial\Omega_i \cap \partial\Omega \\
        \frac{\partial c}{\partial t}-C(v_{ij}, w, c) = 0, \quad \frac{\partial w}{\partial t}-R(v_{ij},w)=0&
    \end{cases}
\end{equation}
with conductivity coefficients $\sigma_i$ in $\Omega_i$, outward normals $n_i$ on $\partial \Omega_i$ and membrane capacitance for unit area $C_m$ on the membrane surface. $v_{ij} = u_i - u_j$ describes the transmembrane voltage, i.e. the discontinuities of the electric potentials between two neighboring subdomains and $F(v_{ij}, c, w)$ stands for either ionic current $I_{\text{ion}}(v_{ij},c,w)$ or gap junction current $G(v_{ij})$, depending on whether the two neighbors are both cells or one is the extracellular domain. We assume $G(v)$ to be linear in the potential jumps $v$ and note that generally, $F(v_{ij}, c, w) = - F(v_{ji}, c, w)$. The last two terms model the ion flow with ordinary differential equations describing the time evolution of ion concentrations $c$ and gating variables $w$. More details on derivation and analysis of the EMI model can be found in \cite{EMI-1, EMI-4, veneroni2006micro}.

We note that since the solution of (\ref{eq:emi}) is only unique up to a constant, we require a zero average on the extracellular solution $u_0$. Additionally, we mention that we consider a splitting strategy in time for the solution of (\ref{eq:emi}), first solving the ionic model with jumps $v_{ij}$ known from the previous time step and then updating the model with the solutions $c$ and $w$ and solving it for the electric potential. For brevity, we will from now on write $F(v_{ij}) \coloneqq F(v_{ij}, c, w)$, omitting $c$ and $w$.

\subsection{Weak Formulation}

On each subdomain $\Omega_i$, integrating the first equations in (\ref{eq:emi}) by parts and substituting in the second equation on the cell membrane, the $i$-th sub-problem reads: find $u_i \in H^1(\Omega_i)$ such that for all $\phi_i \in H^1(\Omega_i)$ the following holds:
\begin{equation*}
    \begin{split}
        0 & = -\int_{\Omega_i}\text{div}(\sigma_i\nabla u_i) \phi_idx \\
        & = \int_{\Omega_i}\sigma_i\nabla u_i \nabla\phi_idx - \int_{\partial \Omega_i} n_i^T\sigma_i\nabla u_i \phi_ids \\
        & = \int_{\Omega_i}\sigma_i\nabla u_i \nabla\phi_idx + \sum_{i \neq j}\int_{F_{ij}} \big(C_m\frac{\partial v_{ij}}{\partial t} + F(v_{ij})\big)\phi_ids.
    \end{split}
\end{equation*} 

Summing up the contributions of all subdomains, we get the global problem
\[\sum_{i = 0}^N\int_{\Omega_i}\sigma_i\nabla u_i\nabla\phi_idx + \frac{1}{2} \sum_{i = 0}^N\sum_{j \neq i}\int_{F_{ij}}\big(C_m\frac{\partial[\![u]\!]_{ij}}{\partial t} + F([\![u]\!]_{ij})\big)[\![\phi]\!]_{ij}ds = 0,\]
where $[\![u]\!]_{ij} = u_i - u_j$ denotes the jump in value of the electric potential $u_i$ and its neighboring $u_j$ from the subdomain $\Omega_j$ along the boundary face $F_{ij}\subset \partial\Omega_i$.

\subsection{Space and Time Discretization}

Let $V_i(\overline{\Omega}_i)$ be the regular finite element space of piece-wise linear functions in $\overline{\Omega}_i$ and define the global finite element space as $V(\Omega) \coloneqq V_0(\overline{\Omega}_0)\times\cdots\times V_N(\overline{\Omega}_N)$. Similar as in \cite{sarkis-3D} we denote:
\begin{itemize}
    \item The \textbf{subdomain face} shared between subdomains $\Omega_i$ and $\Omega_j$ is symbolized by $\overline{F}_{ij} \coloneqq \partial \Omega_i \cap \partial \Omega_j$. Conversely, the subdomain face shared between subdomains $\Omega_j$ and $\Omega_i$ is symbolized by $\overline{F}_{ji} \coloneqq \partial \Omega_j \cap \partial \Omega_i$. We note that geometrically speaking, $\overline{F}_{ij}$ and $\overline{F}_{ji}$ are identical, but to allow for different triangulations on either subdomain \cite{sarkis-3D}, we treat them separately.
    \item $\F_i^0$ describes the set of indices $j$ for which $\Omega_i$ and $\Omega_j$ share a face $F_{ij}$ with non-vanishing two-dimensional measure.
    \item $\N_x$ refers to the set of indexes of subdomains with $x$ in the closure of the subdomain.
\end{itemize}
Just like in \cite{2D-proof}, we consider an implicit-explicit (IMEX) time discretization scheme, treating the diffusion term implicitly and the reaction term explicitly. We split the time interval $[0,T]$ into $K$ intervals. With $\tau = t^{k + 1} - t^k, k = 0,\dots,K$ we derive the following scheme:
\begin{equation*}
    \begin{split}
        \frac{1}{2}\sum_{i = 0}^N\sum_{j \in \F_i^0}\int_{F_{ij}}C_m\frac{[\![u^{k + 1}]\!]_{ij} - [\![u^k]\!]_{ij}}{\tau}[\![\phi]\!]_{ij}ds + \sum_{i = 0}^N\int_{\Omega_i}\sigma_i\nabla u_i^{k+1}\nabla\phi_idx \\
        =-\frac{1}{2}\sum_{i=0}^N\sum_{j \in \F_i^0}\int_{F_{ij}}F([\![u^k]\!]_{ij})[\![\phi]\!]_{ij}ds.
    \end{split}
\end{equation*}

Rearranging the terms such that we only have $[\![u^{k+1}]\!]_{ij}$ and $u_i^{k+1}$ on the left hand side, we get
\begin{equation}
\begin{split}
\frac{1}{2}\sum_{i=0}^N\sum_{j\in \F_i^0}\int_{F_{ij}}C_m[\![u^{k+1}]\!]_{ij}[\![\phi]\!]_{ij}ds + \tau \sum_{i=0}^N\int_{\Omega_i}\sigma_i\nabla u_i^{k+1}\nabla\phi_idx \\
=\frac{1}{2}\sum_{i=0}^N\sum_{j\in \F_i^0}\int_{F_{ij}}C_m[\![u^k]\!]_{ij}[\![\phi]\!]_{ij}ds - \frac{1}{2}\tau\sum_{i=0}^N\sum_{j \in \F_i^0}\int_{F_{ij}}F([\![u^k]\!]_{ij})[\![\phi]\!]_{ij}ds.
\end{split}
\end{equation}

This now lets us define the following local (bi-)linear forms:

\begin{equation}
\label{eq:bilinear-forms}
\begin{split}
    a_i(u_i, \phi_i) & \coloneqq \int_{\Omega_i} \sigma_i \nabla u_i \nabla \phi_i dx, \\
    p_i(u_i, \phi_i) & \coloneqq \frac{1}{2} \sum_{j \neq i} \int_{F_{ij}} C_m [\![u]\!]_{ij}[\![\phi]\!]_{ij} ds, \\
    f_i(\phi_i) & \coloneqq \frac{1}{2} \sum_{j \neq i} \int_{F_{ij}} (C_m [\![u^k]\!]_{ij}[\![\phi]\!]_{ij} - \tau F([\![u^k]\!]_{ij})[\![\phi]\!]_{ij}) ds, \\
    d_i(u_i, \phi_i) & \coloneqq \tau a_i(u_i, \phi_i) + p_i(u_i, \phi_i).
\end{split}
\end{equation}

The global problem now reads: Find $u = \{u_i\}_{i=0}^N \in V(\Omega)$ such that
\begin{equation} \label{eq:bilinear-global}
    d_h(u, \phi) = f(\phi), \quad \forall \phi = \{\phi_i\}_{i=0}^N \in V(\Omega),
\end{equation}

where $d_h(u, \phi) \coloneqq \sum_{i=0}^Nd_i(u_i,\phi_i) = \sum_{i=0}^N(\tau a_i(u_i,\phi_i) + p_i(u_i,\phi_i))$ and $f(\phi) \coloneqq \sum_{i=0}^Nf_i(\phi_i)$. With local stiffness matrices $A_i$ and mass matrices $M_i$, (\ref{eq:bilinear-global}) corresponding to $a_i$ and $p_i$, respectively, can be written in matrix form:
\begin{equation} \label{eq:matrix-form}
    Ku=f, \quad \text{where } K = \sum_{i=0}^NK_i, \quad K_i = \tau A_i + M_i.
\end{equation}
\section{BDDC and Function Spaces} \label{spaces}

In this section, we will first construct the function spaces necessary for our preconditioner, extending the theory in \cite{2D-proof} to the three-dimensional case. Next, we will construct a family of BDDC preconditioners with different classes of primal constraints, building on results from \cite{sarkis-2D, sarkis-2D-deluxe, sarkis-3D,2D-proof} before concluding the section with some technical tools required for the proof of a condition number bound in the next section.

\subsection{Function Spaces}
\label{sec:spaces}
Considering a family of partitions such that each subdomain is the union of shape-regular conforming finite elements, we define the global interface $\Gamma$ as the set of points belonging to at least two subdomains:
\[\Gamma \coloneqq \bigcup_{i = 0}^N \Gamma_i, \quad \Gamma_i \coloneqq \overline{\partial\Omega_i\setminus\partial\Omega}.\]
As in \cite{2D-proof}, we consider discontinuous Galerkin type discretizations in order to correctly treat jumps of the electric potentials along the cell membrane and gap junctions of each myocyte, which means that the degrees of freedom on the interface $\Gamma$ will have a multiplicity depending on the number of subdomains that share them.

We further assume that the finite elements are of diameter $h$ and that all subdomains are shape-regular with a characteristic diameter $H_i$. We denote $H = \max_i H_i$.

Considering the multiplicity of degrees of freedom on the interface, we denote with $\Omega_i' = \overline{\Omega}_i \cup \bigcup_{j \in \F_i^0} \overline{F}_{ji}$ the union of nodes in $\Omega_i$ and on faces $\overline{F}_{ji} \subset \partial \Omega_j$, $j \in \F_i^0$ and we define the according local finite element spaces
\[W_i(\Omega_i') \coloneqq V_i(\overline{\Omega}_i) \times \prod_{j \in \F_i^0} W_i(\overline{F}_{ji}).\]
Here, $W_i(\overline{F}_{ji})$ is the trace of the space $V_j(\overline{\Omega}_j)$ on $F_{ji} \subset \partial \Omega_j$ for all $j \in \F_i^0$. See \Cref{fig:fe-space} for a visualization of $W_i(\Omega_i')$. Each function $u_i \in W_i(\Omega_i')$ in this space can then be written as 
\begin{equation}
    u_i = \{(u_i)_i, \{(u_i)_j\}_{j \in \F_i^0} \}
    \label{u_1}
\end{equation}
where $(u_i)_i$ and $(u_i)_j$ are the restrictions of $u_i$ to $\overline{\Omega}_i$ and $\overline{F}_{ji}$, respectively.

\begin{figure}
    \centering
    \includegraphics[width=.7\textwidth]{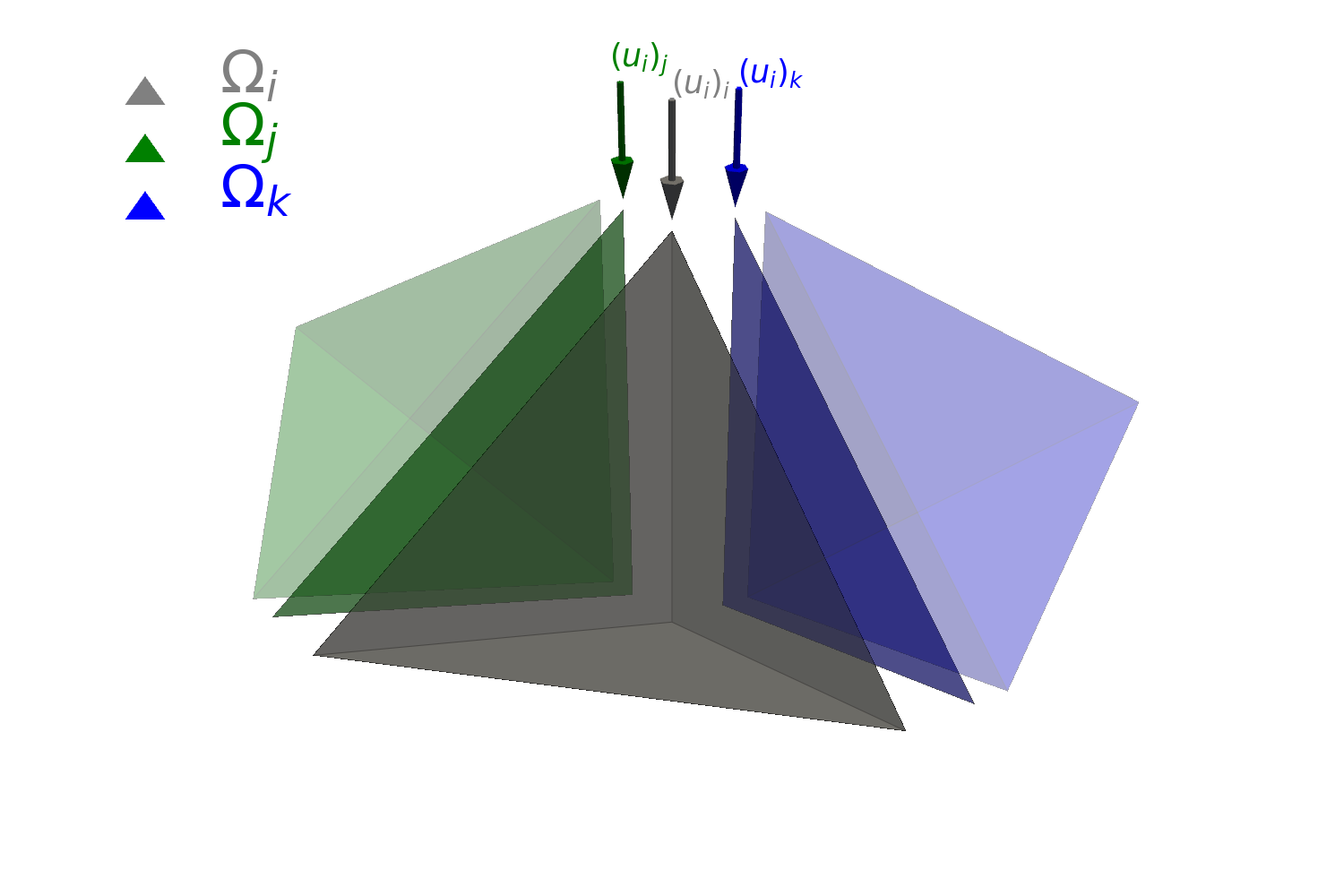}
    \caption{Schematic visualization of the FE space for a tetrahedral substructure $\Omega_i$ (grey) with two neighboring substructures. $u_i \in W(\Omega_i')$ will consist of $(u_i)_i$ on the substructure $\Omega_i$ and of $(u_i)_j$ and $(u_i)_k$, the traces of $V_{\{j, k\}}(\Omega_{\{j, k\}})$ on the faces $F_{ji}$ (green) and $F_{ki}$ (blue), respectively.}
    \label{fig:fe-space}
\end{figure}

Again as in \cite{2D-proof}, we partition $W_i(\Omega_i')$ into its interior part $W_i(I_i)$ and the finite element trace space $W_i(\Gamma_i')$, so
\[W_i(\Omega_i') = W_i(I_i) \times W_i(\Gamma_i')\]
where $\Gamma_i' = \Gamma_i \cup \{\bigcup_{j \in \F_i^0} \overline{F}_{ji}\}$ denotes the local interface nodes in $\Omega_i'$. This allows a representation of (\ref{u_1}) as
\begin{equation}
    u_i = (u_{i,I}, u_{i,\Gamma'}),
\end{equation}
where $u_{i,I}$ represents the values of $u_i$ at the interior nodes on $I_i$ and $u_{i, \Gamma'}$ denotes the values at the nodes on $\Gamma_i'$. We consider the product spaces
\[W(\Omega') \coloneqq \prod_{i = 0}^N W_i(\Omega_i'), \quad W(\Gamma') \coloneqq \prod_{i = 0}^N W_i(\Gamma_i').\]
Here, $u \in W(\Omega')$ means that $u = \{u_i\}_{i = 0}^N$ with $u_i \in W_i(\Omega_i')$ and similarly $u_{\Gamma'} \in W(\Gamma')$ means that $u_{\Gamma'} = \{u_{i, \Gamma'}\}_{i = 0}^N$ with $u_{i, \Gamma'} \in W_i(\Gamma_i')$ where $\Gamma' = \prod_{i = 0}^N \Gamma_i'$ denotes the global broken interface.

\textbf{Subdomain edges} will be denoted by $E_{ijk} \coloneqq \partial F_{ij} \cap \partial F_{ik}$ for two faces $F_{ij}$ and $F_{ik}$ of $\Omega_i$. Similar to the faces, we will treat the three geometrically identical edges $E_{ijk}$, $E_{jik}$ and $E_{kij}$ separately and we define $\E_i^0 \coloneqq \{(j, k)|E_{ijk} \text{ is an edge of } \Omega_i\}$.
\begin{remark}
These edges are geometrically identical because in the EMI model we assume that at each point in space, at most two myocytes are contiguous. This assumption implies that at each point we need to deal with at most three substructures, i.e. the two contiguous myocytes and the extracellular space. From a practical viewpoint, even when considering multiple extracellular domains meeting at an edge on the boundary of a myocyte (as in \Cref{results}, due to computational balancing needs), 
at most three substructures have elements actually contributing to a given degree of freedom. For edges shared only between extracellular domains, the theory from continuous BDDC applies, and therefore for the current analysis we consider the extracellular space to be one substructure.
\end{remark}

Finally we define \textbf{Subdomain vertices} as $\V_i \coloneqq \{\cup_{(j,k) \in \E_i^0} \partial E_{ijk}\}$ and $\V_i'$ as the union of $\V_i$ with all vertices from other subdomains that $\Omega_i$ has a share in. We say that $u = \{u_i\}_{i = 0}^N \in W(\Omega')$ is continuous at the corners $\V_i$ if
\[(u_i)_i(x) = (u_j)_i(x) \text{ at all } x \in \V_i \text{ for all } j \in \N_x.\]

\begin{definition}
    The discrete harmonic extension $\HH_i'$ in the sense of $d_i$ as defined in (\ref{eq:bilinear-forms}) is defined as
    \begin{equation*}
    \HH_i': W_i(\Gamma') \rightarrow W_i(\Omega_i'), 
    \begin{cases}
        d_i(\HH_i'u_{i,\Gamma'}, v_i) = 0 & \forall v_i \in W_i(\Omega_i') \\
        \HH_i'u_{i,\Gamma'} = (u_i)_i & \text{on } \partial \Omega_i \\
        \HH_i'u_{i,\Gamma'} = (u_i)_j & \text{on } F_{ji} \subset \partial \Omega_j \\
        & \text{and on } E_{j, \E} \subset \partial \Omega_j.
    \end{cases}
    \end{equation*}
\end{definition}

With a notion of faces, edges, and vertices, we can define the function spaces relevant for our preconditioner as in \cite{dd-book, sarkis-3D}:



\begin{definition}[Subspaces $\widetilde{W}(\Omega')$ and $\widetilde{W}(\Gamma')$]
\label{sec2:VEF}
We define the space $\widetilde{W}(\Omega')$ as the subspace of functions $u = \{u_i\}_{i=0}^N \in W(\Omega')$ for which the following conditions hold for all $0 \leq i \leq N$:
\begin{itemize}
    \item $u$ is continuous at all corners $\V_i$.
    \item On all edges $E_{ijk}$ for $(j,k) \in \E_i^0$
    \[(\overline{u}_i)_{i,E_{ijk}} = (\overline{u}_j)_{i,E_{ijk}} = (\overline{u}_k)_{i, E_{ijk}}.\]
    \item On all faces $F_{ij}$ for $j \in \F_i^0$
    \[(\overline{u}_i)_{i,F_{ij}} = (\overline{u}_j)_{i,F_{ij}}.\]
\end{itemize}

Here, 
\[(\overline{u}_j)_{i,E_{ijk}} = \frac{1}{|E_{ijk}|} \int_{E_{ijk}}(u_j)_i ds, \quad (\overline{u}_j)_{i,F_{ij}} = \frac{1}{|F_{ij}|} \int_{F_{ij}}(u_j)_i ds.\]
We denote with $\widetilde{W}(\Gamma')$ the subspace of $\widetilde{W}(\Omega')$ of functions which are discrete harmonic in the sense of $\HH_i'$.
\end{definition}

\begin{figure}
    \centering
    \includegraphics[width=.5\textwidth]{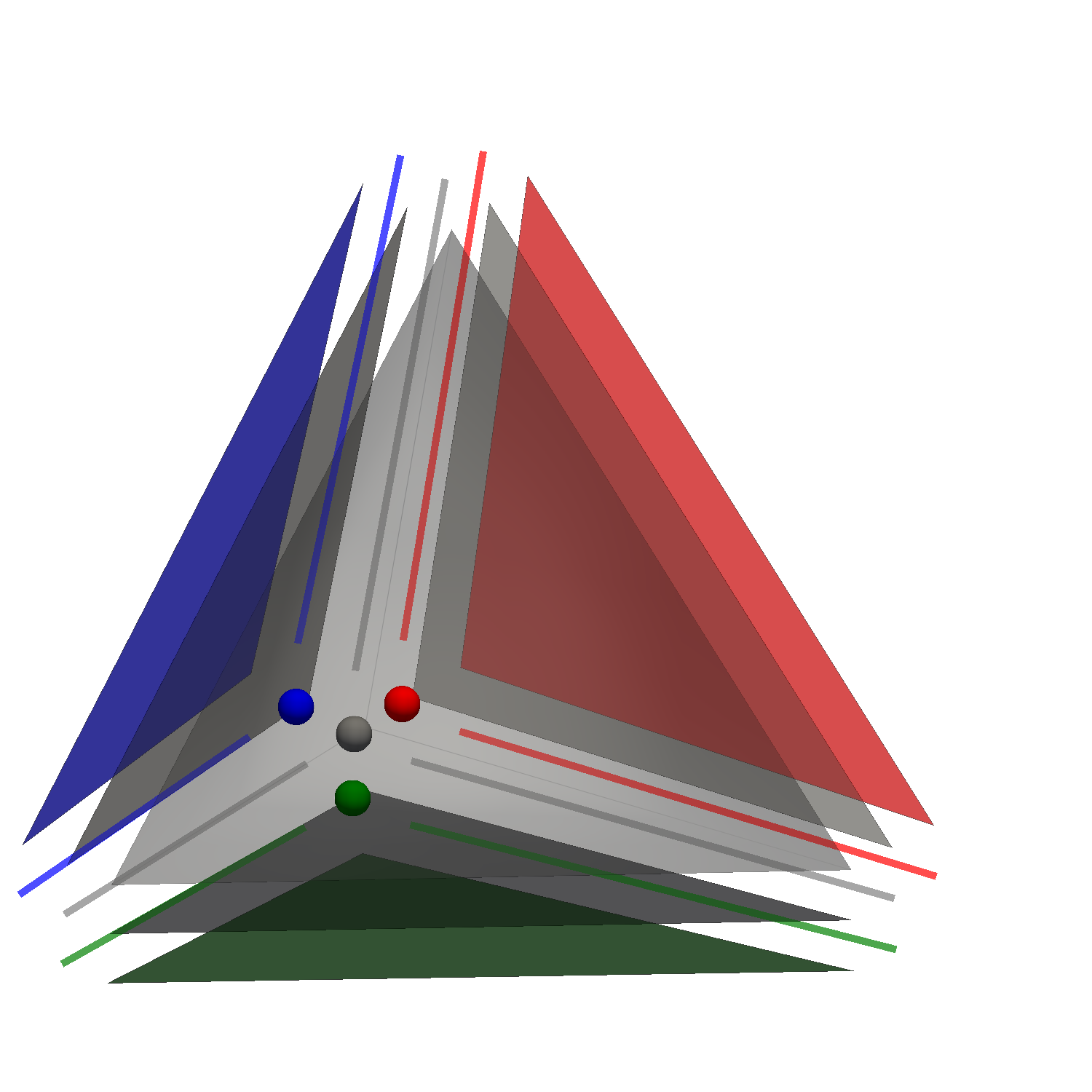}
    \caption{Representation of primal constraints connected to a tetrahedral substructure $\Omega_i$ (grey) surrounded by three neighbors (red, green and blue), with the fourth face intersecting with $\partial \Omega$ (the global Neumann boundary). For each face, edge and vertex, one primal constraint per involved substructure is created and on each of them, the according averaging constraints are imposed. In this particular example, $\Omega_i$ will contribute to the primal space with 6 face average, 9 edge average and 4 vertex constraints.}
    \label{fig:primal-constraints}
\end{figure}

As indicated in a simplified way by \Cref{fig:primal-constraints}, the primal space $\widetilde{W}(\Omega')$ can grow in dimension very quickly, which increases the computational cost of setting up a BDDC preconditioner dramatically, making it desirable to reduce the amount of primal constraints. We therefore also consider the following setting, removing the face averages from the primal space:

\begin{definition}[Subspaces $\widetilde{W}_{VE}(\Omega')$ and $\widetilde{W}_{VE}(\Gamma')$]
\label{sec2:VE}
We define the space \newline
$\widetilde{W}_{VE}(\Omega')$ as the subspace of functions $u = \{u_i\}_{i=0}^N \in W(\Omega')$ for which the following conditions hold for all $0 \leq i \leq N$:
\begin{itemize}
    \item $u$ is continuous at all corners $\V_i$.
    \item On all edges $E_{ijk}$ for $(j,k) \in \E_i^0$
    \[(\overline{u}_i)_{i,E_{ijk}} = (\overline{u}_j)_{i,E_{ijk}} = (\overline{u}_k)_{i, E_{ijk}}.\]
\end{itemize}

We define $\widetilde{W}_{VE}(\Gamma')$ in the same way as $\widetilde{W}(\Gamma')$ above.
\end{definition}

A function $u \in \widetilde W(\Omega')$ (or $u \in \widetilde{W}_{VE}(\Omega')$) can be represented as $u = (u_I, u_{\Delta}, u_{\Pi})$ where $I = \prod_{i = 0}^N I_i$ represents the degrees of freedom on \textit{interior} nodes, $\Pi$, which we will call \textit{primal}, denotes degrees of freedom at the vertices $\V_i'$ and the average face and edge values (or only average edge values in the case of $\widetilde{W}_{VE}(\Omega')$). $\Delta$ refers to the remaining nodal degrees of freedom on $\Gamma_i' \setminus \V_i'$ with zero interface averages for the involved interfaces. We will refer to them as \textit{dual}. 
A more exhaustive explanation of such representation can be found in \cite{dd-book}.

Introducing the spaces
\[W_{\Delta}(\Gamma') = \prod_{i = 0}^N W_{i, \Delta}(\Gamma_i') \quad\text{and}\quad \widetilde W_{\Pi}(\Gamma')\]
where $W_{i,\Delta}(\Gamma_i')$ are the local spaces associated with the dual degrees of freedom and $\widetilde W_{\Pi}(\Gamma')$ refers to the space associated with the primal degrees of freedom, we can decompose $\widetilde W(\Gamma')$ into
\[\widetilde W(\Gamma') = W_{\Delta}(\Gamma') \times \widetilde W_{\Pi}(\Gamma')\]
and get the representation
\[u_{\Gamma'} \in \widetilde W(\Gamma'), \quad u_{\Gamma'} = (u_{\Delta}, u_{\Pi})\]
with $u_{\Pi} \in \widetilde W_{\Pi}(\Gamma')$ and $u_{\Delta} = \{u_{i, \Delta}\}_{i = 0}^N \in W_{\Delta}(\Gamma')$. Note that we can write $u_{i, \Delta}$ as
\[u_{i, \Delta} = \{\{u_{i, F_{ij}}, u_{i, F_{ji}}\}_{j \in \F_i^0}, \{u_{i, E_{ijk}}, u_{i, E_{jik}}, u_{i, E_{kij}}\}_{(j, k) \in \E_i^0}\}\]
where $u_{i, F}$ is the restriction of $u_{i, \Delta}$ to the face $F$, and $u_{i, E}$ is the restriction of $u_{i, \Delta}$ to the edges $E$. For convenience, \Cref{tab:spaces} gives an overview of the above mentioned function spaces.


{
\renewcommand{\arraystretch}{1.2}
\begin{table}[h]
    \centering
    \begin{tabular}{l p{9cm}}
         \hline 
         \textbf{Space Symbol} & \textbf{Short Description} \\\hline
         $V_i(\overline{\Omega}_i)$ & local FE space on $\overline{\Omega}_i$ \\
         $W_i(\overline{F}_{ji})$ & trace of $V_j(\overline{\Omega}_j)$ on the face $F_{ji}$ between \newline subdomains $\Omega_i$ and $\Omega_j$ \\
         $W_i(\Omega_i')$ & local FE space including duplicated face degrees of freedom \\
         $W_i(I_i)$ & interior part of $W_i(\Omega_i')$ \\
         $W_i(\Gamma_i')$ & part of $W_i(\Omega_i')$ on the broken interface $\Gamma_i'$ \\
         $W(\Omega'), W(\Gamma')$ & global product spaces of the $W_i(\Omega_i')$ and $W_i(\Gamma_i')$ \\
         $\widetilde{W}(\Omega')$ & functions $u \in W(\Omega')$ continuous on vertices \newline with edge and face average constraints \\
         $\widetilde{W}_{VE}(\Omega')$ & functions $u \in W(\Omega')$ continuous on vertices \newline with edge but no face average constraints \\
         $\widetilde{W}(\Gamma'),\widetilde{W}_{VE}(\Gamma')$ & subspaces of $\widetilde{W}(\Omega')$, $\widetilde{W}_{VE}(\Omega')$ of functions \newline which are discrete harmonic in the sense of $\HH_i'$ \\
         $W_{i, \Delta}(\Gamma_i')$ & local space associated with nodal degrees of freedom on\newline $\Gamma_i'\setminus \V_i'$ with zero interface averages for the involved interfaces \\
         $W_{\Delta}(\Gamma')$ & global product space of the $W_{i, \Delta}(\Gamma_i')$ \\
         $\widetilde{W}_{\Pi}(\Gamma')$ & space of degrees of freedom associated with primal constraints (vertex values and face / edge averages)\\\hline
    \end{tabular}
    \caption{A list of the function spaces mentioned in \Cref{sec:spaces} with short descriptions.}
    \label{tab:spaces}
\end{table}
}
\subsection{Schur Bilinear Form, Restriction and Scaling Operators}

Assuming that in algebraic form, we can write the local problems as $K_i'u_i = f_i$, we can order the degrees of freedom such that the local matrices read
\begin{equation}
    \label{local-matrix}
    K_i' = \begin{bmatrix} K_{i, II}' & K_{i, I\Gamma'}' \\ K_{i, \Gamma' I}' & K_{i, \Gamma' \Gamma'}' \end{bmatrix}.
\end{equation}
By eliminating the interior degrees of freedom (\textit{static condensation}), our preconditioner will work only on the unknowns on the interface $\Gamma'$. In order to do this, we need the local Schur complement systems
\[S_i' \coloneqq K_{i, \Gamma' \Gamma'}' - K_{i, \Gamma' I}'(K_{i, II}')^{-1}K_{i, I \Gamma'}'\]
with which we define the \textit{unassembled} global Schur complement matrix \[S' = \text{diag}[S_0',\dots,S_N'].\] 

Let $R_{\Gamma'}^{(i)}: W(\Omega') \rightarrow W_i(\Gamma_i')$ denote the restriction operators returning the local interface components and define $R_{\Gamma'} \coloneqq \sum_{i = 0}^N R_{\Gamma'}^{(i)}$. The global Schur complement matrix is then given by $\widehat S_{\Gamma'} = R_{\Gamma'}^TS'R_{\Gamma'}$.

Hence, instead of solving the global linear system $Ku = f$, we can first eliminate the interior degrees of freedom to retrieve a right-hand side $\widehat f_{\Gamma'}$ on the interface $\Gamma'$, then solve the Schur complement system
\[\widehat S_{\Gamma'} u_{\Gamma'} = \widehat f_{\Gamma'}\]
and use the solution $u_{\Gamma'}$ on the interface to recover the interior solution as
\[u_{i,I} = (K_{i,II}')^{-1}(f_{i,I} - K_{i,I\Gamma'}'u_{\Gamma'}).\]

The Schur bilinear form can now be defined as 
\[d_i(\HH_i'u_{i,\Gamma'},\HH_i'v_{i,\Gamma'}) = v_{i,\Gamma'}^TS_i'u_{i,\Gamma'} = s_i'(u_{i,\Gamma'},v_{i,\Gamma'})\]
and it has the property
\begin{equation}
    \label{schur-min-property}
    s_i'(u_{i,\Gamma'},u_{i,\Gamma'}) = \min_{v_i|_{\partial \Omega_i \cap \Gamma'} = u_{i,\Gamma'}} d_i(v_i,v_i)
\end{equation}
which allows us to work with discrete harmonic extensions rather than functions only defined on $\Gamma'$.

The function spaces from the previous chapter will be equipped with the following restriction operators:
\begin{equation*}
    \begin{split}
        R_{i, \Delta}: W_{\Delta}(\Gamma')\rightarrow W_{i, \Delta}(\Gamma'), \quad \quad & R_{\Gamma' \Delta}: W(\Gamma') \rightarrow W_{\Delta}(\Gamma'), \\
        R_{i, \Pi}: \widetilde W_{\Pi}(\Gamma') \rightarrow W_{i, \Pi}(\Gamma_i'), \quad \quad & R_{\Gamma' \Pi}: W(\Gamma') \rightarrow\widetilde W_{\Pi}(\Gamma').
    \end{split}
\end{equation*}
We further define the direct sums $R_{\Delta} = \oplus R_{i, \Delta}, R_{\Pi} = \oplus R_{i, \Pi}$ and $\widetilde R_{\Gamma'} = R_{\Gamma' \Pi} \oplus R_{\Gamma' \Delta}$.

For the EMI model, we consider $\rho$-scaling for the dual variables. 
Other scalings, such as stiffness or deluxe scaling, could be considered in the future. For this theoretical analysis, we concentrate on the simplest case of $\rho$-scaling.
For $x \in \overline{\Omega}_i$ it is defined by the pseudoinverses
\begin{equation}
    \label{rho-scaling}
    \delta_i^{\dag}(x) \coloneqq \frac{\sigma_i}{\sum_{j \in \N_x} \sigma_j}.
\end{equation}
For $x \notin \overline{\Omega}_i$ we define $\delta_i^{\dag}(x) = 0$. We note that the $\delta_i^{\dag}$ form a partition of unity, so $\sum_{i = 0}^N \delta_i^{\dag} (x) = 1$ for all $x \in \Omega$.

Recall the following inequality ((6.19) in \cite{dd-book}), it will be an important tool later in this paper:
\begin{equation}
    \label{sigma-ineq}
    \sigma_i (\delta_j^\dag(x))^2 \leq \min\{\sigma_i, \sigma_j\} \quad \forall j \in \N_x.
\end{equation}

We define the local scaling operators on each subdomain $\Omega_i$ as the diagonal matrix
\begin{equation}
    D_i \coloneqq \text{diag}(\delta_i^{\dag}),
\end{equation}
so the $i$-th scaling matrix contains the coefficients (\ref{rho-scaling}) evaluated on the nodal points of $\Omega_i$ along the diagonal. With the scaling operators, we can define scaled local restriction operators
\[R_{i, D, \Gamma'} \coloneqq D_iR_{i, \Gamma'},\quad R_{i,D,\Delta} \coloneqq R_{i,\Gamma'\Delta}R_{i,D,\Gamma'},\]
$R_{D,\Delta}$ as the direct sum of the $R_{i,D,\Delta}$ and finally the global scaled restriction operator 
\[\widetilde R_{D,\Gamma'} \coloneqq R_{\Gamma'\Pi} \oplus R_{D,\Delta}R_{\Gamma'\Delta}.\]

\subsection{The BDDC preconditioner}
Introduced in \cite{dohrmann}, Balancing Domain Decomposition by Constraints (BDDC) is a two-level preconditioner for the Schur complement system $\widehat S_{\Gamma'}u_{\Gamma'} = \widehat f_{\Gamma'}$. Partitioning the degrees of freedom in each subdomain $\Omega_i$ into interior ($I$), dual ($\Delta$) and primal ($\Pi$) degrees of freedom, we can further partition (\ref{eq:matrix-form}) into
\[K_i' = \begin{bmatrix}
    K_{i,II}' & K_{i,I\Delta}' & K_{i,I\Pi}' \\
    K_{i,\Delta I}' & K_{i,\Delta \Delta}' & K_{i, \Delta \Pi}' \\
    K_{i,\Gamma I}' & K_{i, \Gamma\Delta}' & K_{i, \Gamma\Gamma}'
\end{bmatrix}.\]

The action of the BDDC preconditioner is now given by 
\[M^{-1}_{\text{BDDC}}x = \widetilde R^T_{D, \Gamma'}(\widetilde S_{\Gamma'})^{-1}\widetilde R_{D, \Gamma'}x,\quad \widetilde S_{\Gamma'} = \widetilde R_{\Gamma'} S' \widetilde R_{\Gamma'}^T,\]
where the inverse of $\widetilde S^{-1}_{\Gamma'}$ does not have to be computed explicitly but can be evaluated with Block-Cholesky elimination via
\[\widetilde S^{-1}_{\Gamma'} = \widetilde R^T_{\Gamma' \Delta} \big( \sum_{i = 0}^N \begin{bmatrix} 0 & R^T_{i, \Delta} \end{bmatrix}\begin{bmatrix} K_{i, II}' & K_{i, I\Delta}' \\ K_{i, \Delta I}' & K_{i, \Delta \Delta}'\end{bmatrix}^{-1}\begin{bmatrix}0 \\ R_{i, \Delta}\end{bmatrix} \big) \widetilde R_{\Gamma'\Delta} + \Phi S_{\Pi\Pi}^{-1}\Phi^T.\]
The first term above is the sum of independently computed local solvers on each subdomain $\Omega_i'$, and the second term is a coarse solver for the primal variables (which is computed independently of the local solvers), where
\[\Phi = R^T_{\Gamma'\Pi} - R^T_{\Gamma'\Delta}\sum_{i = 0}^N\begin{bmatrix} 0 & R^T_{i, \Delta} \end{bmatrix}\begin{bmatrix} K_{i, II}' & K_{i, I\Delta}' \\ K_{i, \Delta I}' & K_{i, \Delta \Delta}'\end{bmatrix}^{-1}\begin{bmatrix}K_{i,I\Pi} \\ R_{i, \Delta\Pi}\end{bmatrix}R_{i,\Pi},\]
\[S_{\Pi\Pi} = \sum_{i=0}^N R_{i,\Pi}^T\big(K_{i, \Pi\Pi}' - \begin{bmatrix} K_{i,\Pi I} & K_{i,\Pi\Delta}' \end{bmatrix}\begin{bmatrix} K_{i, II}' & K_{i, I\Delta}' \\ K_{i, \Delta I}' & K_{i, \Delta \Delta}'\end{bmatrix}^{-1}\begin{bmatrix}K_{i,I\Pi} \\ R_{i, \Delta\Pi}\end{bmatrix}\big)R_{i,\Pi}.\]

\subsection{Technical Tools and Assumptions}
As in \cite{2D-proof}, we will utilize the following Lemma that can be proven considering the continuity and coercivity of the standard Laplacian bilinear form $a_i$:
\begin{lemma}
    For the bilinear form $d_i(u_i, v_i) = \tau a_i(u_i, v_i) + p_i(u_i, v_i)$ with $a_i$ and $p_i$ as defined in (\ref{eq:bilinear-forms}), the following bounds hold:
    \[d_i(u_i, u_i) \leq \tau \sigma_M |u_i|^2_{H^1(\Omega_i)} + \sum_{j \neq i} \frac{C_m}{2}\|u_i - u_j\|^2_{L^2(F_{ij})},\]
    \[d_i(u_i, u_i) \geq \tau \sigma_m|u_i|^2_{H^1(\Omega_i)} + \sum_{j \neq i} \frac{C_m}{2}\|u_i - u_j\|^2_{L^2(F_{ij})},\]
    for all $u_i \in V_i(\Omega_i)$ with $\sigma_m$ and $\sigma_M$ being the minimum and maximum values of the coefficients $\sigma_i$, respectively.
\end{lemma}

We assume that the cardiac domain $\Omega \subset \R^3$ is subdivided into substructures $\Omega_i \subset \R^3$ which have Lipschitz-continuous boundaries. We will work with Sobolev spaces on subsets $\Lambda \subset \partial \Omega_i$ of the subdomain boundaries which have non-vanishing two-dimensional (faces) or one-dimensional (edges) measure and and are relatively open to $\partial \Omega_i$. We will mostly consider the space $H^{1/2}(\Lambda)$ of functions $u \in L^2(\Lambda)$ with finite semi-norm $|u|_{H^{1/2}(\Lambda)} < \infty$ and norm 
\[\|u\|^2_{H^{1/2}(\Lambda)} = \|u\|^2_{L^2(\Lambda)} + |u|^2_{H^{1/2}(\Lambda)} < \infty.\]

We will also use the set of functions in $H^{1/2}(\Lambda)$ which extend to zero from $\Lambda$ to $\partial \Omega_i$ by the extension operator $\E_{\text{ext}}$
\[\E_{\text{ext}}: \Lambda \rightarrow \partial \Omega_i,\quad \E_{\text{ext}}u = \begin{cases}
    0 & \text{on } \partial\Omega_i \setminus \Lambda \\
    u & \text{on } \Lambda.
\end{cases}\]

We will denote this space by
\[H^{1/2}_{00}(\Lambda) = \big\{ u \in H^{1/2}(\Lambda): \E_{\text{ext}}u \in H^{1/2}(\partial\Omega_i)\big\}.\]

\begin{remark}
    For notational convenience, we will write $A \lesssim B$ whenever $A \leq cB$ where $c$ is some constant independent on problem parameters (like e.g. subdomain sizes, mesh size, or conductivity coefficients). Whenever both $A \lesssim B$ and $B \lesssim A$ hold, we will write $A \sim B$.
\end{remark}

The following Lemmas will be used in the theoretical analysis in the next chapter, their proofs can be found in \cite{dd-book} (Lemma 4.17, Lemma 4.19 and Lemma 4.26).

\begin{lemma}
\label{sec2:edge-face}
    Let $\overline{u}_{E_{ijk}}$ be the average value of $u$ over $E_{ijk}$, an edge of the face $F_{ij}$. Then,
    \[\|u\|^2_{L^2(E_{ijk})} \lesssim \bigg( 1 + \log\frac{H}{h} \bigg) \|u\|^2_{H^{1/2}(F_{ij})}\]
    and
    \[\|u - \overline{u}_{E_{ijk}}\|^2_{L^2(E_{ijk})} \lesssim \bigg(1 + \log\frac{H}{h}\bigg) |u|^2_{H^{1/2}(F_{ij})}.\]
\end{lemma}

In short, \Cref{sec2:edge-face} will enable us to bound edge terms by face terms of an adjacent face. Another inequality we will leverage for the edge terms is
\begin{lemma}
\label{sec2:boundary-edge}
    Let $E_{ijk}$ be an edge of a subdomain $\Omega_i$ and let $u \in V^h$. Then, 
    \[|\HH_i'(\Theta_{E_{ijk}} u)|_{H^{1/2}(\partial \Omega_i)}^2 \lesssim \|I^h\Theta_{E_{ijk}}u\|_{L^2(E_{ijk})}^2,\]
    where $\Theta_{E_{ijk}}$ is the characteristic finite element function on the edge $E_{ijk}$ and $I^h$ is the usual finite element interpolant.
\end{lemma}

The following Lemma will be useful for the proof considering reduced primal space only containing vertex and edge average constraints:

\begin{lemma}
    Let $\overline{u}_{\E}$ be the average of $u$ over $\E$, an edge of subdomain $\Omega_i$. Let $H_{\E}$ be the diameter of this edge. Then,
    \[(\overline{u}_{\E})^2 \lesssim \frac{1}{H_{\E}} \|u\|^2_{L^2(\E)}.\]
    \label{edge-average}
\end{lemma}

For face terms, we will use the following inequality:

\begin{lemma}
\label{sec2:face}
    Let $F_{ij}$ be a face of a subdomain $\Omega_i$, let $u \in V^h$ and $\overline{u}_{F_{ij}}$ be the average of $u$ over $F_{ij}$. Then,
    \[\|\Theta_{F_{ij}}\|_{H^{1/2}_{00}(F_{ij})} \lesssim \bigg(1 + \log\frac{H}{h}\bigg)H\]
    and
    \[\|I^h\Theta_{F_{ij}}(u - \overline{u}_{F_{ij}})\|^2_{H^{1/2}_{00}(F_{ij})} \lesssim \bigg(1 + \log\frac{H}{h}\bigg)^2|u|^2_{H^{1/2}(\partial\Omega_i)}.\]
\end{lemma}

Finally, we will consider the following classical, well-known results throughout the proof. They can be found for example in Appendix A of \cite{dd-book}.

\begin{lemma}
    Let $\Lambda \subset \partial\Omega_i$. Then, for $u \in H^{1/2}_{00}$ it holds that 
    \[\|\E_{\text{ext}}u\|^2_{H^{1/2}(\partial \Omega_i)} \lesssim \|u\|^2_{H^{1/2}_{00}(\Lambda)} \lesssim \|\E_{\text{ext}}u\|^2_{H^{1/2}(\partial\Omega_i)}.\]
\end{lemma}

\begin{theorem}[Trace theorem] Let $\Omega_i$ be a polyhedral domain. Then,
\[|u|_{H^{1/2}(\Gamma_i)} \sim |\HH^{\Delta}_iu_{\Gamma}|^2_{H^1(\Omega_i)}.\]

\end{theorem}

\begin{lemma}\label{lemma-A17}
    Let $\Omega \subset \mathbb R^3$ be a Lipschitz continuous polyhedron and let $\mathcal F \subset \partial \Omega$ be one of the faces of $\Omega$ with diameter $H$. Then, there exists a constant $C$ that depends only on the shape of $\mathcal F$ but not on its size, such that
    \begin{equation*}
        \| u \|^2_{L^2 (\mathcal F)} \leq C \, H | u |^2_{H^{1/2} (\mathcal F)},
    \end{equation*}
    if $u \in H^{1/2} (\mathcal F)$ either has a vanishing mean value on $\mathcal F$ or belongs to $H^{1/2}_{00} (\mathcal F)$.
\end{lemma}

\section{Bound for the Jump Operator} \label{sec:proof}
The main tool in the theory of dual-primal iterative substructuring methods is given by the \textbf{jump operator} $P_D:W(\Gamma') \rightarrow W(\Gamma')$, whose action on a given $u \in W(\Gamma')$ is given by $P_D = I - E_D$, where $E_D$ is the averaging operator defined by
\begin{equation}
    (E_Du)_i(x) = \bigg( \sum_{j \in \N_x} \delta_j^{\dag}(x) (u_j)_i(x), \Big\{\sum_{j \in \N_x} \delta_j^{\dag}(x) (u_j)_{j_0}(x) \Big\}_{j_0 \in \F_i^0, j_0 \in (k,l) \in \E_i^0} \bigg).
\end{equation}
For this projection, for $x \in \Gamma_i$ it holds that
\begin{equation*}
    \begin{split}
        ((E_Du)_i)_i(x) & = \sum_{j \in \N_x} \delta_j^{\dag}(x) (u_j)_i(x) = ((E_D u)_{j_0})_i(x) \quad \forall j_0 \in \N_x, \\
        ((E_Du)_i)_{j_0}(x) & = \sum_{j \in \N_x} \delta_j^{\dag}(x) (u_j)_{j_0}(x) = ((E_Du)_{j_0})_{j_0}(x) \quad \forall j_0 \in \F_i^0, j_0 \in (k,l) \in \E_i^0
    \end{split}
\end{equation*}
and it induces the local action of the jump operator for the considered EMI problem as
\begin{equation}
    \begin{split}
    (P_Du)_i(x) = \bigg( & \sum_{j \in \N_x} \delta_j^{\dag}(x)((u_i)_i(x) - (u_j)_i(x)), \\
    & \Big\{ \sum_{j \in \N_x} \delta_j^{\dag}(x)((u_i)_{j_0}(x) - (u_j)_{j_0}(x))\Big\}_{j_0 \in \F_i^0, j_0 \in (k,l) \in \E_i^0} \bigg).
    \end{split}
\end{equation}
We recall that $((E_D u)_i)_i$ and $((E_D u)_i)_{j_0}$ are the restriction of $(E_D u)_i$ to $\overline \Omega_i$ and $\overline F_{j_0 i}$ respectively.

The main contribution of this work is a bound on the norm of this operator with respect to the seminorm $|\cdot|^2_{S_i'}$ induced by the Schur operator:
\[|v|_{S'}^2 \coloneqq \sum_{i = 0}^N |v_i|_{S_i'}^2, \quad \quad |v_i|_{S_i'}^2 \coloneqq v_i^TS_i'v_i, \quad \text{for } v_i \in W(\Gamma_i').\]

\begin{lemma}
\label{lemma-full-primal}
    Let the primal space $\widetilde W_{\Pi}(\Gamma')$ be spanned by the vertex nodal degrees finite element functions and the face and edge averages. If the jump operator $P_D$ is scaled by the $\rho$-scaling as defined in (\ref{rho-scaling}), then
    \begin{equation}
        |P_Du|_{S'}^2 \leq C \bigg(1 + \log \frac{H}{h} \bigg)^2 |u|_{S'}^2
    \end{equation}
    holds for all $u \in \widetilde W(\Gamma')$ with $C$ constant and independent of all parameters of the problem. Here, $h$ is the mesh size.
\end{lemma}

\begin{proof}
    As in \cite{2D-proof}, we need only to consider the local contributions
    \begin{equation*}
        \begin{split}
            v_i & \coloneqq (P_Du)_i \\
            & = \bigg( \sum_{j \in \N_x} I^h \delta_j^{\dag}((u_i)_i - (u_j)_i), \Big\{\sum_{j \in \N_x} I^h \delta_j^{\dag}((u_i)_{j_0} - (u_j)_{j_0})\Big\}_{j_0 \in \F_i^0, j_0 \in (k,l) \in \E_i^0} \bigg) \\
            & = \sum_{\F}I^h(\Theta_{\F} v_i) + \sum_{\E} I^h(\Theta_{\E}v_i) + \sum_{\V} I^h(\Theta_{\V} v_i).
        \end{split}
    \end{equation*}
    Here $\Theta_* = (\theta_*^i, \theta_*^{j_1}, \dots, \theta_*^{j_k})$ is the characteristic finite element function associated to the faces ($k = 1$), edges ($k = 2$) and vertices ($* \in \{\F, \E, \V\}$) and $I^h$ is the usual finite element interpolant
    \[I^h: \mathcal{C}^0(\Omega_i') \rightarrow W_i(\Omega_i').\]
    Since the vertices are in the primal space, the vertex term vanishes and we have to estimate the contributions given by the faces and edges
    \[|v_i|_{S_i'}^2 \lesssim \sum_{\F}|I^h(\Theta_{\F}v_i)|_{S_i'}^2 + \sum_{\E}|I^h(\Theta_{\E}v_i)|_{S_i'}^2.\]

    For $* \in \{\E, \F\}$, the ellipticity property of the bilinear form $d_i$ as defined in (\ref{eq:bilinear-forms}) gives
    \begin{equation}
    \label{first-ineq}
        \begin{split}
            |I^h(\Theta_*v_i)|_{S_i'}^2 & = s_i'(I^h(\Theta_* v_i), I^h(\Theta_* v_i)) = d_i(\HH_i' I^h(\Theta_* v_i), \HH_i' I^h(\Theta_* v_i)) \\
            & \leq \tau \sigma_i |\HH_i^{\Delta} I^h(\Theta_* v_i)|_{H^1(\Omega_i)}^2 + \sum_{j \in \F_i^0} \frac{C_m}{2}\|{I^h}(v_i - v_j)\|_{L^2(F_{ij})}^2.
        \end{split}
    \end{equation}

    For the second term in (\ref{first-ineq}), we use the same argumentation as in \cite{2D-proof}, making use of the explicit representation of $v_i$ and $v_j$ on the face $F_{ij}$ (see \cite{2D-proof} for an analogous in two-dimensions), to get 
    \begin{equation}
    \label{C_m-ineq}
        \|I^h(v_i - v_j)\|_{L^2(F_{ij})} \lesssim |u_i|_{S_i'}^2 + |u_j|_{S_j'}^2.
    \end{equation}

    For the first term, we have to distinguish between edge and face terms.

    \textit{Edge Terms:} Consider the edge $\E = E_{ijk}$. We have to consider three terms contributing to
    \begin{equation*}
        \label{full-edge}
        \sigma_i |\HH_i^{\Delta} (I^h(\Theta_{E_{ijk}} v_i))|_{H^1(\Omega_i)}^2 \lesssim \sum_{l \in \{i,j,k\}} \sigma_i (\delta_l^{\dag})^2 \|\HH_i^{\Delta}(I^h \Theta_{E_{ijk}}((u_i)_i - (u_l)_i))\|_{H^1(\Omega_i)}^2.
    \end{equation*}
    For $l = i$, the term vanishes and for $l \in \{j,k\}$ we get
    \begin{equation*}
        \begin{split}
            & \sigma_i (\delta_l^{\dag})^2 \|\HH_i^{\Delta}(I^h \Theta_{\E}((u_i)_i - (u_l)_i))\|_{H^1(\Omega_i)}^2 \\
            & \lesssim \sigma_i (\delta_l^{\dag})^2 \|I^h \Theta_{\E}((u_i)_i - (u_l)_i)\|_{H^{1 / 2}(\partial \Omega_i)}^2 \\
            & \lesssim \sigma_i (\delta_l^{\dag})^2 \|I^h(\Theta_{\E}((u_i)_i - (\overline{u}_i)_{i, \E}) - \Theta_{\E}((u_l)_i - (\overline{u}_l)_{i, \E}))\|_{L^2(\E)}^2 \\
            & \lesssim \sigma_i \| (u_i)_i - (\overline{u}_i)_{i, \E}\|_{L^2(\E)}^2 + \sigma_l \| (u_l)_i - (\overline{u}_l)_{i, \E}\|_{L^2(\E)}^2 \\
            & \lesssim \bigg( 1 + \log \frac{H}{h}\bigg) \bigg( \sigma_i |(u_i)_i|_{H^{1 / 2}(F_{il_1})}^2 + \sigma_l |(u_l)_i|_{H^{1 / 2}(F_{ll_2})}^2\bigg) \\
            & \lesssim \bigg(1 + \log \frac{H}{h}\bigg) \big(\frac{1}{\tau}|u_i|_{S_i'}^2 + \frac{1}{\tau}|u_{l}|_{S_{l}'}^2\big)
        \end{split}
    \end{equation*}
    where we use \Cref{sec2:boundary-edge}, the equal edge averages, (\ref{sigma-ineq}) and then \Cref{sec2:edge-face} for some face $F_{il_1}$ of $\Omega_i$ and some face $F_{ll_2}$ of $\Omega_l$.
    
    Summing up, we have shown that 
    \begin{equation}
        \label{edge-ineq}
        |I^h\Theta_{E_{ijk}}v_i|_{S_i'}^2 \lesssim \bigg(1 + \log \frac{H}{h}\bigg) \sum_{l \in \{i,j,k\}}|u_l|_{S_l'}^2.
    \end{equation}

    \textit{Face Terms:} In the same way, considering the face $\F = F_{ij}$ between subdomains $\Omega_i$ and $\Omega_j$ we get
    \begin{equation}
    \label{face-ineq}
        \begin{split}
            \sigma_i |\HH_i^{\Delta}(I^h(\Theta_{F_{ij}} v_i))&|_{H^1(\Omega_i)}^2 \\
            & \lesssim \sigma_i (\delta_j^{\dag})^2 \|\HH_i^{\Delta}(I^h \Theta_{F_{ij}}((u_i)_i - (u_j)_i))\|_{H^1(\Omega_i)}^2 \\
            & \lesssim \sigma_i (\delta_j^{\dag})^2 \|I^h \Theta_{F_{ij}}((u_i)_i - (u_j)_i)\|_{H^{1 / 2}(\Gamma_i)}^2 \\
            & \lesssim \sigma_i (\delta_j^{\dag})^2 \|I^h\Theta_{F_{ij}}((u_i)_i - (u_j)_i)\|_{H_{00}^{1 / 2}(F_{ij})}^2 \\
            & \lesssim \sigma_i \|I^h\Theta_{F_{ij}}((u_i)_i - (\overline u_i)_{i, F_{ij}})\|_{H_{00}^{1 / 2}(F_{ij})}^2 \\&\qquad \qquad + \sigma_j \|I^h\Theta_{F_{ij}}((u_j)_i - (\overline u_j)_{i, F_{ij}})\|_{H_{00}^{1 / 2}(F_{ij})}^2 \\
            & \lesssim \bigg( 1 + \log \frac{H}{h} \bigg)^2 \big(\sigma_i |(u_i)_i|_{H^{1 / 2}(F_{ij})}^2 + \sigma_j |(u_j)_i)|_{H^{1 / 2}(F_{ij})}^2 \big) \\
            & \lesssim \bigg( 1 + \log \frac{H}{h}\bigg)^2 \big(\frac{1}{\tau}|u_i|_{S_i'}^2 + \frac{1}{\tau}|u_j|_{S_j'}^2 \big)
        \end{split}
    \end{equation}
    where we use (\ref{sigma-ineq}) and \Cref{sec2:face}.

    Adding up over all substructures, faces and edges, (\ref{C_m-ineq}), (\ref{edge-ineq}) and (\ref{face-ineq}) finally give the inequality
    \[
    |P_Du|_{S'}^2 \lesssim \bigg( 1 + \log \frac{H}{h} \bigg)^2 |u|_{S'}^2.
    \]
\end{proof}

We can show a similar result considering the space $\widetilde{W}_{VE}$, reaching the same asymptotic behavior with fewer primal constraints:

\begin{lemma}
\label{lem:VE}
    Let the primal space $\widetilde W_{\Pi}(\Gamma')$ be spanned by the vertex nodal degrees finite element functions and the edge averages. If the jump operator $P_D$ is scaled by the $\rho$-scaling as defined in (\ref{rho-scaling}), then
    \begin{equation}
        |P_Du|_{S'}^2 \leq C \bigg(1 + \log \frac{H}{h} \bigg)^2 |u|_{S'}^2
    \end{equation}
    holds for all $u \in \widetilde W_{VE}(\Gamma')$ with $C$ constant and independent of all parameters of the problem.
\end{lemma}

\begin{proof}
    We proceed just as in the proof of \Cref{lemma-full-primal} for edge terms. For face terms, the averages $(\overline u_i)_{i, F_{ij}}$ and $(\overline u_j)_{i, F_{ij}}$ will now generally be different. Therefore, following the argumentation of Lemma 6.36 in \cite{dd-book} we reach that
    \begin{equation*}
    \label{face-ineq-no-av}
        \begin{split}
            & \sigma_i (\delta_j^{\dag})^2 \|I^h\Theta_{F_{ij}}((u_i)_i - (u_j)_i)\|_{H_{00}^{1 / 2}(F_{ij})}^2 \\
            & = \sigma_i (\delta_j^{\dag})^2 \|I^h\Theta_{F_{ij}}(((u_i)_i - (\overline u_i)_{i, F_{ij}}) - ((u_j)_i - (\overline u_j)_{i, F_{ij}}) \\
            & \quad + ((\overline u_i)_{i, F_{ij}} - (\overline u_j)_{i, F_{ij}}))\|_{H_{00}^{1 / 2}(F_{ij})}^2 \\
            & \lesssim \sigma_i \|I^h\Theta_{F_{ij}}((u_i)_i - (\overline u_i)_{i, F_{ij}})\|_{H_{00}^{1 / 2}(F_{ij})}^2 + \sigma_j \|I^h\Theta_{F_{ij}}((u_j)_i - (\overline u_j)_{i, F_{ij}})\|_{H_{00}^{1 / 2}(F_{ij})}^2 \\
            & \quad + \min\{\sigma_i, \sigma_j\} \|\Theta_{F_{ij}} ((\overline u_i)_{i, F_{ij}} - (\overline u_j)_{i, F_{ij}})\|^2_{H^{1/2}_{00}(F_{ij})}.
        \end{split}
    \end{equation*}
    The first two terms we can estimate exactly as in (\ref{face-ineq}) by
    \[\big(\sigma_i |(u_i)_i|_{H^{1 / 2}(F_{ij})}^2 + \sigma_j |(u_j)_i)|_{H^{1 / 2}(F_{ij})}^2 \big).\]
    We still have to estimate the third term, $\|\Theta_{F_{ij}} ((\overline u_i)_{i, F_{ij}} - (\overline u_j)_{i, F_{ij}})\|^2_{H^{1/2}_{00}(F_{ij})}$. Since all edge averages are primal, we choose an edge $\E$ of the face $F_{ij}$ and have $(\overline{u}_i)_{i, \E} = (\overline{u}_j)_{i, \E}$. This gives
    \[|(\overline u_i)_{i, F_{ij}} - (\overline u_j)_{i, F_{ij}}|^2 \lesssim |(\overline u_i)_{i, F_{ij}} - (\overline{u}_i)_{i, \E}|^2 + |(\overline u_j)_{i, F_{ij}} - (\overline{u}_j)_{i, \E}|^2.\]
    We only consider the first term and proceed in exactly the same way for the second. Utilizing \Cref{edge-average}, we get
    \[|(\overline u_i)_{i, F_{ij}} - (\overline{u}_i)_{i, \E}|^2 = |(\overline{(\overline u_i)_{i, F_{ij}} - (u_i)_i})_{i, \E}|^2 \lesssim \frac{1}{H}\|(\overline u_i)_{i, F_{ij}} - (u_i)_i\|^2_{L^2(\E)}, \]
    \textcolor{cyan}{
    and \Cref{sec2:edge-face} leads to
    \begin{equation*}
        |(\overline u_i)_{i, F_{ij}} - (\overline{u}_i)_{i, \E}|^2
        \lesssim \frac{1}{H} \bigg(1 + \log \frac{H}{h} \bigg) \|(\overline u_i)_{i, F_{ij}} - (u_i)_i\|^2_{H^{1/2}(F_{ij})}.
    \end{equation*}
    Applying \Cref{lemma-A17}, we ultimately bound the above line with
    \begin{equation}
        \frac{1}{H} \bigg(1 + \log\frac{H}{h}\bigg) |(\overline u_i)_{i, F_{ij}} - (u_i)_i|^2_{H^{1/2}(F_{ij})}
        \label{edge-minus-face}
    \end{equation}
    }
    Combining (\ref{edge-minus-face}) with \Cref{sec2:face} it finally follows that
    \[\|\Theta_{F_{ij}} ((\overline u_i)_{i, F_{ij}} - (\overline u_j)_{i, F_{ij}})\|^2_{H^{1/2}_{00}(F_{ij})} \lesssim \bigg( 1 + \log\frac{H}{h}\bigg)^2\big(|(u_i)_i|^2_{H^{1/2}(F_{ij)}} + |(u_j)_i|^2_{H^{1/2}(F_{ij})}\big)\]
    and we can combine the partial results and follow the rest of the argumentation in the proof of \Cref{lemma-full-primal}.
\end{proof}

With these results, we finally receive the following condition number bounds for the BDDC-preconditioned linear systems which can be proven as in \cite{dd-book}:
\begin{theorem}
    Let $M^{-1}_{BDDC}$ be the BDDC preconditioner built as stated above and $K$ the according linear system. Assume that the jump operator is scaled by the $\rho$-scaling, as considered in \Cref{lemma-full-primal} and \Cref{lem:VE}. Then the condition number $\kappa(M^{-1}K) = \frac{\lambda_{\max}}{\lambda_{\min}}$ is bounded by
    \[\kappa(M^{-1}K) \leq C \bigg(1 + \log\frac{H}{h}\bigg)^2,\]
    where $\lambda_{\max}$ and $\lambda_{\min}$ are the maximum and minimum Eigenvalues of $M^{-1}K$ and $C$ is a constant independent of $h, H$, and the values of $\sigma_i.$ This bound holds both if we choose the primal constraints as in \Cref{lemma-full-primal} and as in \Cref{lem:VE}.
\end{theorem}
\section{Numerical Study} \label{results}

We complement the theoretical analysis with experiments. For this, we consider an artificial, repetitive geometry that splits a cube into two subdomains, an intracellular subdomain in the center with connections to the outside via all faces of the cubes and an extracellular domain around it (see~\Cref{fig:mesh}). This way, each intracellular subdomain has both an interface to an extracellular subdomain (via a cell membrane) and to other intracellular subdomains (via gap junctions). For this study, we consider a linear gap junction $\frac{v_{ij}}{R_g}$ where $R_g= 4.5\times 10^{-4}$ for the interfaces between adjacent intracellular subdomains and the Aliev-Panfilov ionic model \cite{aliev} for the gating variables between intra- and extracellular subdomains. If not stated otherwise, the conductivity coefficients $\sigma_i$ are fixed to $3 \frac{\text{mS}}{\text{cm}}$ for intra- and $20 \frac{\text{mS}}{\text{cm}}$ for extra-cellular subdomains. To improve load balancing, we decompose the extracellular domain into subdomains using regular continuous finite elements. On the interfaces between those subdomains, we do not need to consider any discontinuities. We use the BDDC implementation in the software library Ginkgo \cite{ginkgo} as a preconditioner for a Conjugate Gradient (CG) method, where we always choose a zero initial guess. The stopping criterion evaluates the $L^2$-norm of the residual against a preset threshold of $10^{-6}$, except for \Cref{tab:weak_scaling}, where we consider the absolute residual norm of $10^{-8}$. Iteration matrices and related right-hand sides are computed using the finite element library Kaskade 7 \footnote{https://www.zib.de/other/kaskade7/}. The matrices $K_i$ are computed as explained in \cite{2D-proof}. The times reported refer to the solving time of the preconditioned solver. Numerical results are given for the solution of the linear system at time step 0.01 ms.
All tests have been performed on the CPU partition of the EuroHPC machine Karolina\footnote{https://www.it4i.cz/en/infrastructure/karolina} on compute nodes with two AMD Zen 2 EPYC™ 7H12 CPUs, totalling 128 CPU cores and 256 GB of main memory per node. 

\begin{figure}[ht!]
    \centering
    \includegraphics[width=.3\textwidth]{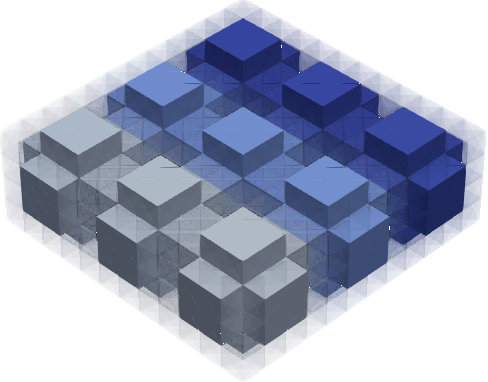}
    \includegraphics[width=.3\textwidth]{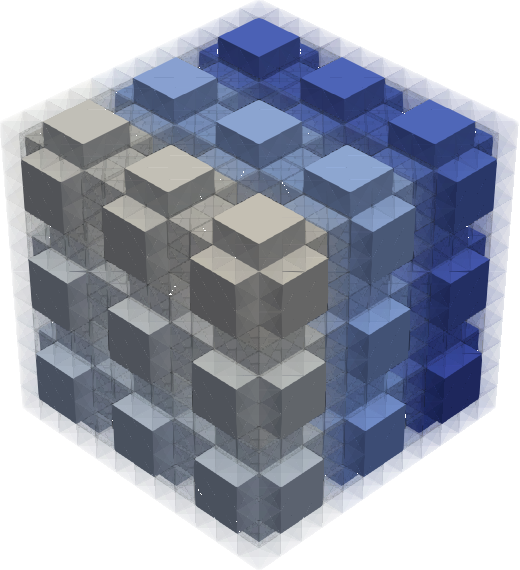}
    \includegraphics[width=.3\textwidth]{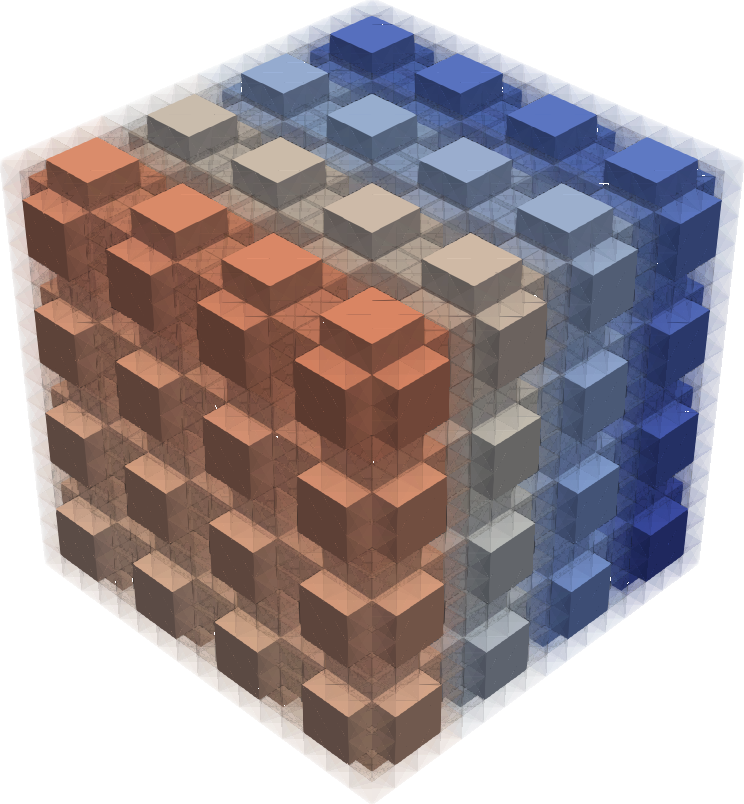}
    \caption{Repetitive test geometry with 3x3x1 (left), 3x3x3 (middle) and 4x4x4 (right) cells. Each cell is an intracellular subdomain inside a cube that can be stacked in all dimensions, resulting in a mesh where all intracellular subdomains have interfaces with extracellular space via a cell membrane model and with other intracellular domains via a linear gap junction.}
    \label{fig:mesh}
\end{figure}

\subsection{Scalability}
For a weak scaling study, we consider the linear system in an EMI model simulation at the same time step (0.01 ms) for a growing number of subdomains, starting with $3\times 3\times 1$ cells and reaching up to $7\times 7\times 7$ cells. The observed convergence behavior aligns with the theory: as we increase the number of subdomains but leave $\frac{H}{h}$ constant, the number of iterations needed to converge as well as the condition number estimate remain roughly constant, see \Cref{tab:weak_scaling}.

\begin{table}[ht]
    \centering
    \begin{tabular}{ccccccccccc}
         \hline 
         \multirow{2}{1cm}{\#cells} & \multirow{2}{0.8cm}{\#SD} & \multirow{2}{1cm}{GD} & \multicolumn{4}{c}{VEF} & \multicolumn{4}{c}{VE} \\
         &&& it & $\kappa$ & CD & time & it & $\kappa$ & CD & time \\\hline
         3x3x1 & 18 & 4493 & 7 & 2.3158 & 73 & 2.9 & 14 & 17.4721 & 40 & 2.5\\
         3x3x2 & 36 & 8718 & 7 & 2.2712 & 207 & 3.3 & 16 & 12.2331 & 123 & 2.8\\
         3x3x3 & 54 & 12943 & 7 & 2.2587 & 341 & 3.8 & 15 & 11.9005 & 206 & 3.3\\
         4x4x4 & 128 & 30209 & 7 & 2.2605 & 919 & 5.4 & 15 & 11.7062 & 567 & 6\\
         5x5x5 & 250 & 58461 & 7 & 2.2602 & 1929 & 5.8 & 15 & 11.3865 & 1204 & 6.2\\
         6x6x6 & 432 & 100405 & 7 & 2.2602 & 3491 & 9.1 & 15 & 11.6549 & 2195 & 6.6\\
         7x7x7 & 686 & 158747 & 7 & 2.2602 & 5725 & 15.6 & 15 & 11.6975 & 3618 & 15.5 \\\hline
    \end{tabular}
    \caption{Weak scalability for an increasing number of cells from $3\times 3\times 1$ to $7\times 7\times 7$. Each cube is discretized with 1024 tetrahedra, i.e. 512 for each intra- and extracellular subdomain. We report the number of subdomains (SD), the global dimension (GD) of the linear problem, the number of preconditioner CG iterations (it), a condition number estimate ($\kappa$) computed with the Lanczos estimate, the dimension of the coarse problem (CD), and the time needed for a preconditioner application in ms. The stopping criterion tolerance for this test is a residual norm of $10^{-8}$.}
    \label{tab:weak_scaling}
\end{table}

To evaluate the stability of the method, we consider a setup where we run the solver with random right-hand-side vectors, in order to confirm that (as expected from the theory) the Krylov method convergence is independent of the right-hand side and initial guess. For this study, we generate 100 different right-hand-side vectors filled with random values in $(-1,1)$ for each of the test cases and record the iteration count needed to converge to a relative residual norm tolerance of $10^{-6}$. 
\Cref{fig:rand_rhs} confirms the expectation that the choice of right-hand side does not impact the convergence of our method significantly. In terms of compute time, we see two major jumps: the first when inter-node communication over the network is needed starting at 128 subdomains, and the second when the solution of the coarse problem starts dominating the runtime. The latter effect can possibly be alleviated by employing a second level of BDDC on the coarse problem to improve scalability.

\begin{figure}
    \centering
    \begin{subfigure}[b]{0.4\textwidth}
        \includegraphics[width=\textwidth]{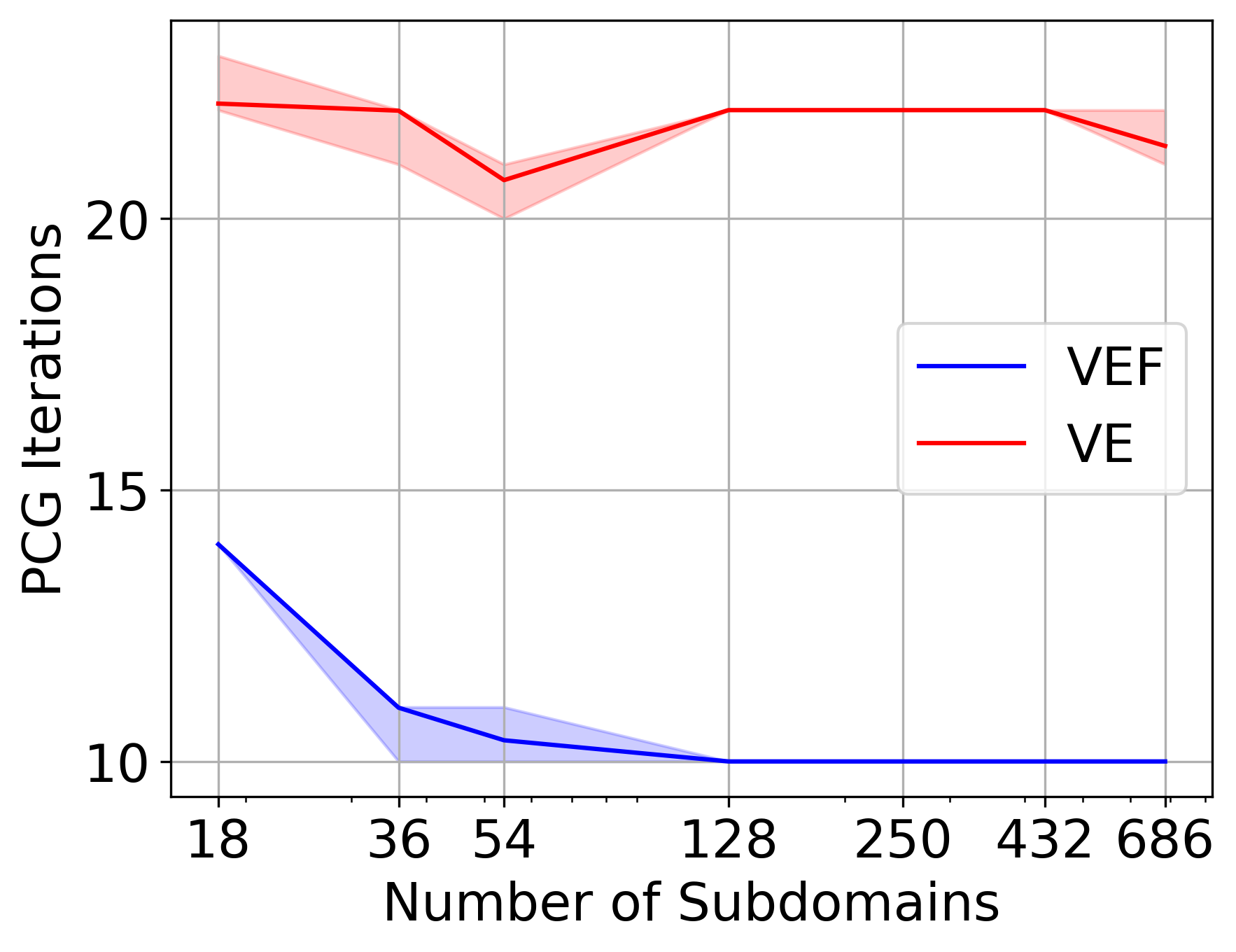}
    \end{subfigure}
    \begin{subfigure}[b]{0.4\textwidth}
        \includegraphics[width=\textwidth]{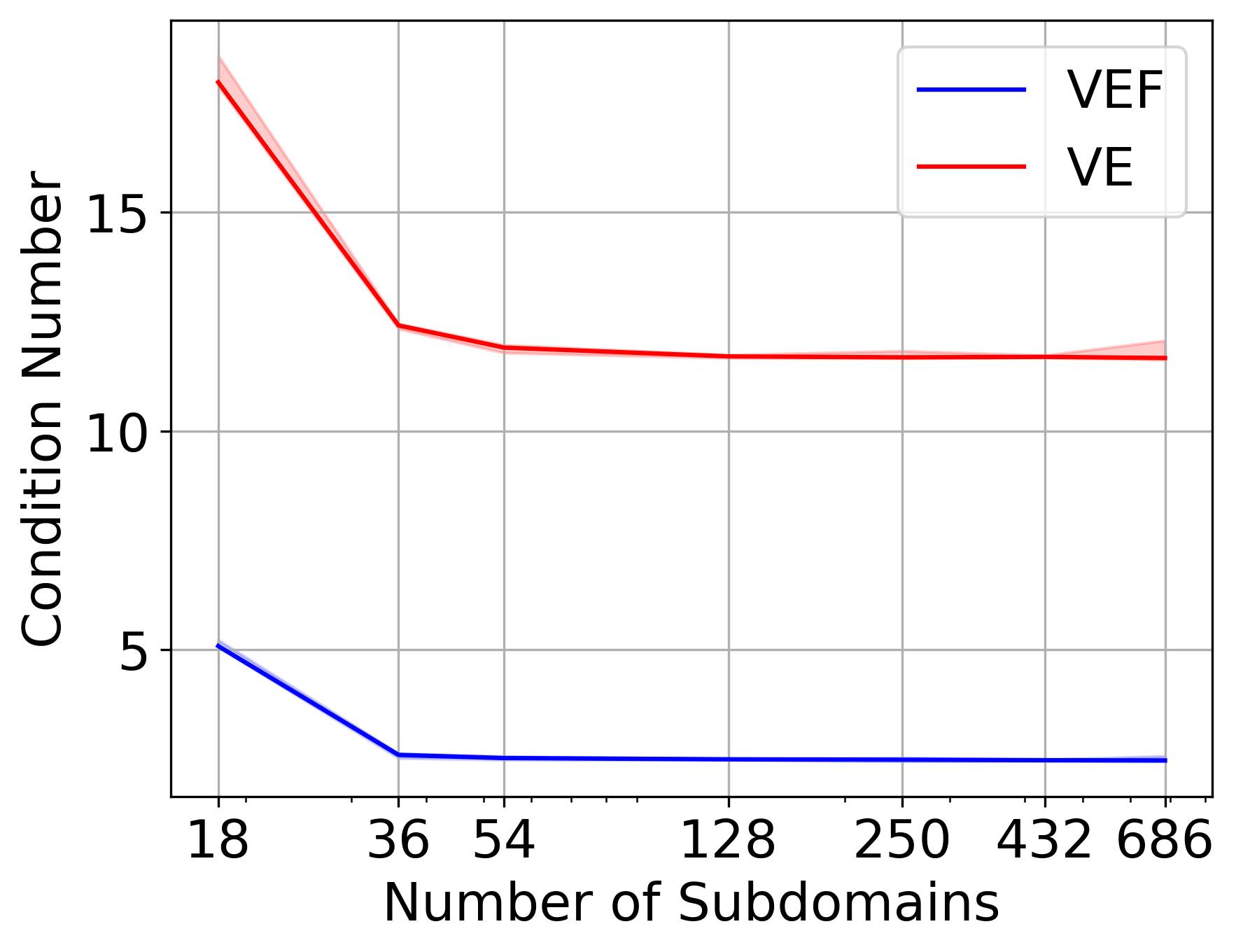}
    \end{subfigure}
    \caption{Iterations needed to converge to a relative residual tolerance of $10^{-6}$ (left) and condition number estimates (right) for random right-hand side vectors. The results colored in blue consider a full primal space containing vertex values as well as edge and face averages. The results colored in red consider only vertex values and edge averages in the primal space. The solid lines show the mean over 100 different random right-hand sides, the colored areas represent the range of iterations or condition numbers for each test case, respectively.}
    \label{fig:rand_rhs}
\end{figure}

\subsection{Robustness w.r.t.\ conductivity coefficients}
The theoretical results obtained in \Cref{sec:proof} reveal that the condition number of the preconditioned operator is bounded independently of the conductivity coefficients $\sigma_i$. In order to evaluate this experimentally, we show in \Cref{fig:rand_sigma} the convergence behavior for random conductivity coefficients in the extracellular and intracellular subdomains. As the extracellular subdomains together represent one continuous space, we assign the same coefficient to all of them, while each of the intracellular subdomains is assigned a random conductivity coefficient, resulting in constant conductivity coefficients within each subdomain. One can observe the same behavior as for fixed $\sigma_i$ with a slightly wider range of needed iterations.

\begin{figure}
    \centering
    \begin{subfigure}[b]{0.4\textwidth}
        \includegraphics[width=\textwidth]{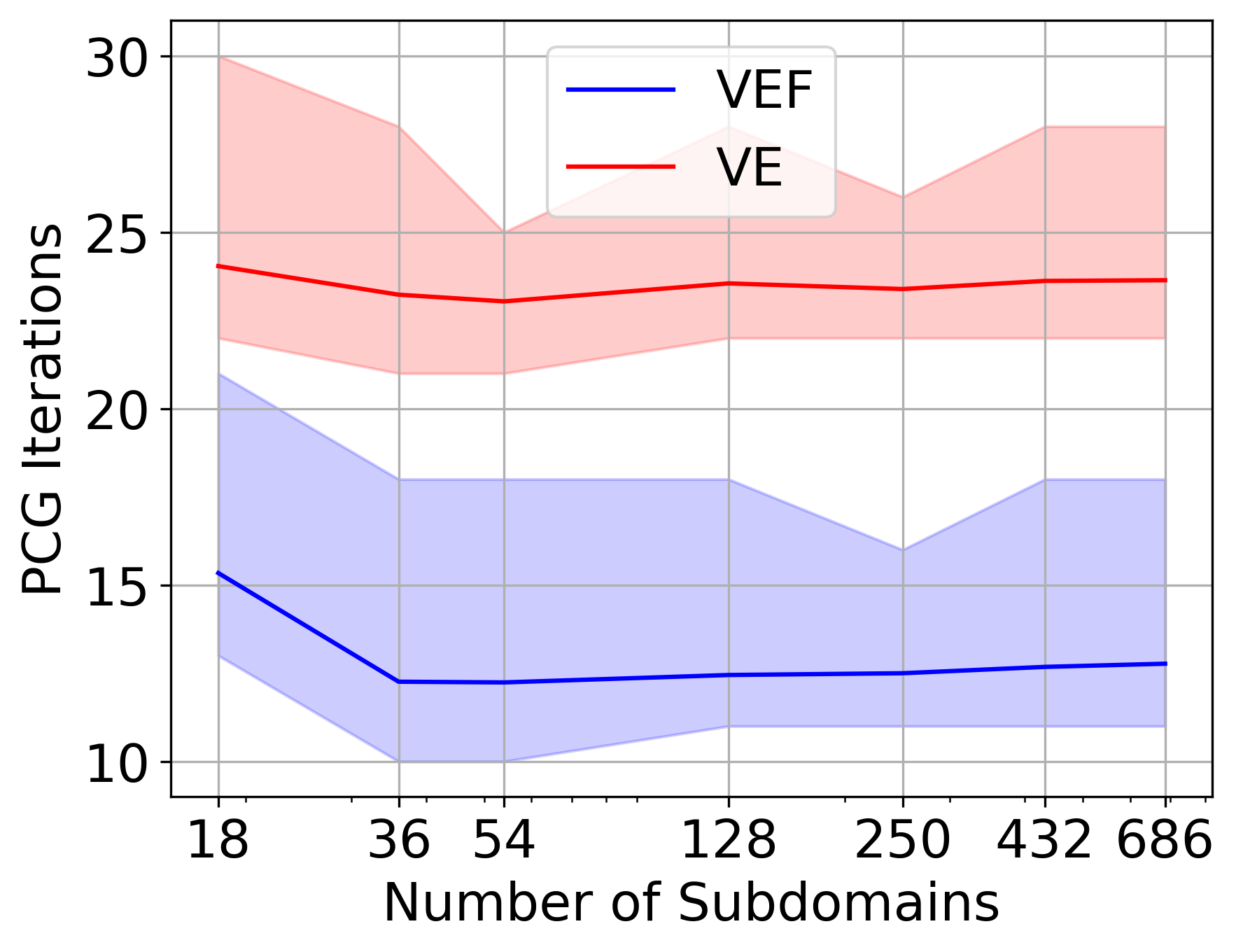}
    \end{subfigure}
    \begin{subfigure}[b]{0.4\textwidth}
        \includegraphics[width=\textwidth]{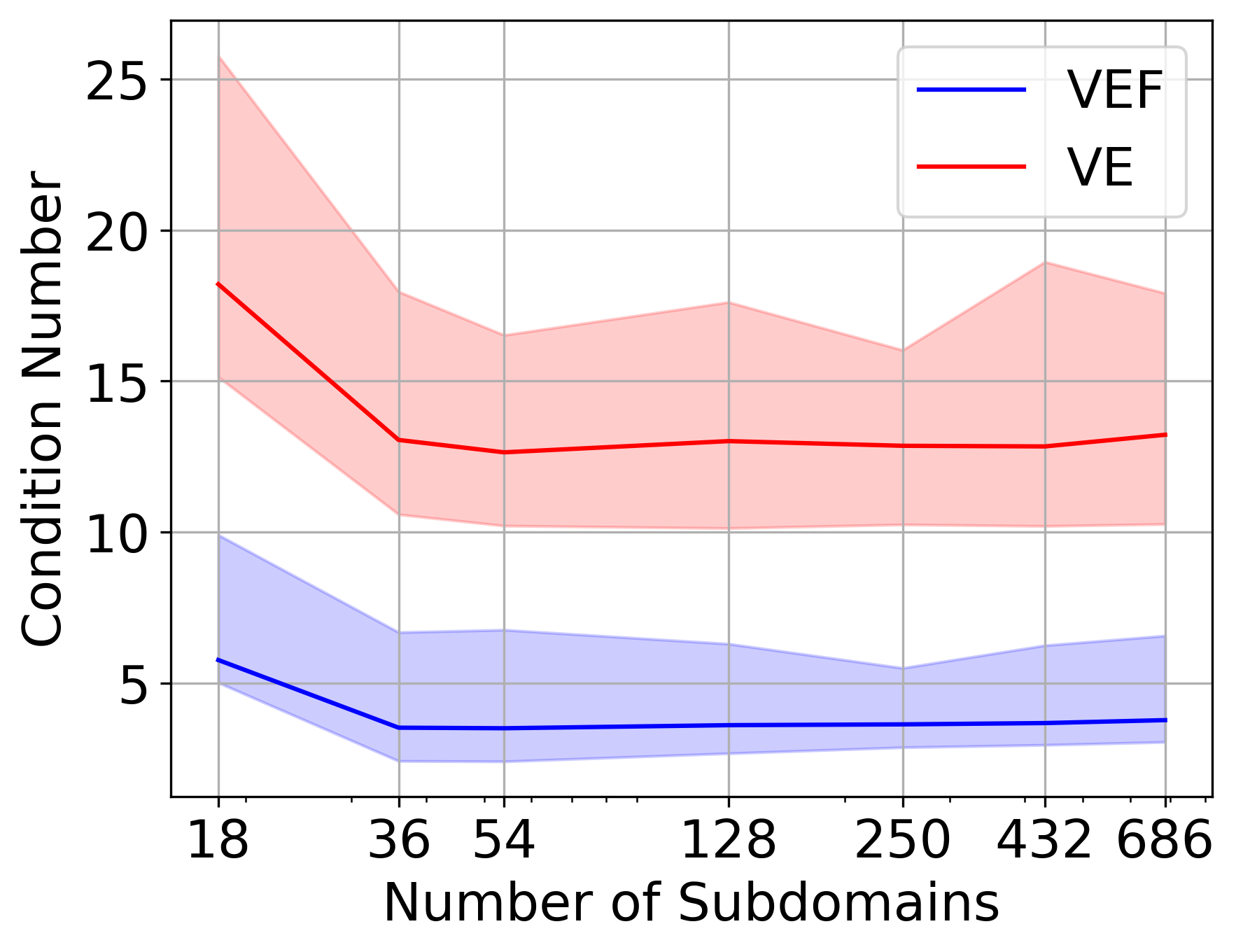}
    \end{subfigure}
    \caption{Iterations needed to converge to a relative residual tolerance of $10^{-6}$ (left) and condition number estimates (right) for random conductivity coefficients \newline $\sigma_i \in (1, 20) \frac{\text{mS}}{\text{cm}}$. For each test case, we generate the preconditioner 100 times with different, random conductivity coefficients, the solid lines show the mean of the iterations needed to converge and the condition number estimates.}
    \label{fig:rand_sigma}
\end{figure}

\subsection{Optimality tests}
In this section, we evaluate the poly-logarithmic convergence behavior of the preconditioned linear operator. For this, we refine a mesh of $3\times 3\times 3$ cells in order to increase the value of $\frac{H}{h}$. In each refinement level, the number of degrees of freedom doubles in each dimension, so $\frac{H}{h}$ also increases by a factor of 2. \Cref{fig:refinement} reveals that the iteration count increases as expected. From the theoretical results, we expect the condition number to grow asymptotically with $(1 + \log\frac{H}{h})^2$ (dashed blue line).

The estimated condition number grow aligns with this expectation. For the reduced coarse space only consisting of edge and vertex constraints, the condition number growth aligns more with linear growth (red dotted line), while flattening out with increasing refinement level.
This linear dependency on $H/h$ instead of the polylogarithmic one could be explained by the fact that we might not be in the asymptotic range for the reduced coarse space, and by the highly non-convex shape of the substructures, in particular the extracellular domain, which could impact on the convergence.

To verify our assumption, we simplify our non-convex geometry to the
meshes shown in \Cref{fig:simple-meshes}. This way, we intend to observe the asymptotic behavior of the condition number also for the reduced coarse space in \Cref{fig:refinement-simple}. We note that while the estimated condition number of the preconditioned operator reflects the theoretical asymptotic bounds perfectly for the full coarse space and for the reduced coarse space, this is only obvious for the simplest mesh. The actual convergence behavior by iterations reflects poly-logarithmic growth in all cases.

\begin{figure}
    \centering
    \begin{subfigure}[b]{0.4\textwidth}
        \includegraphics[width=\textwidth]{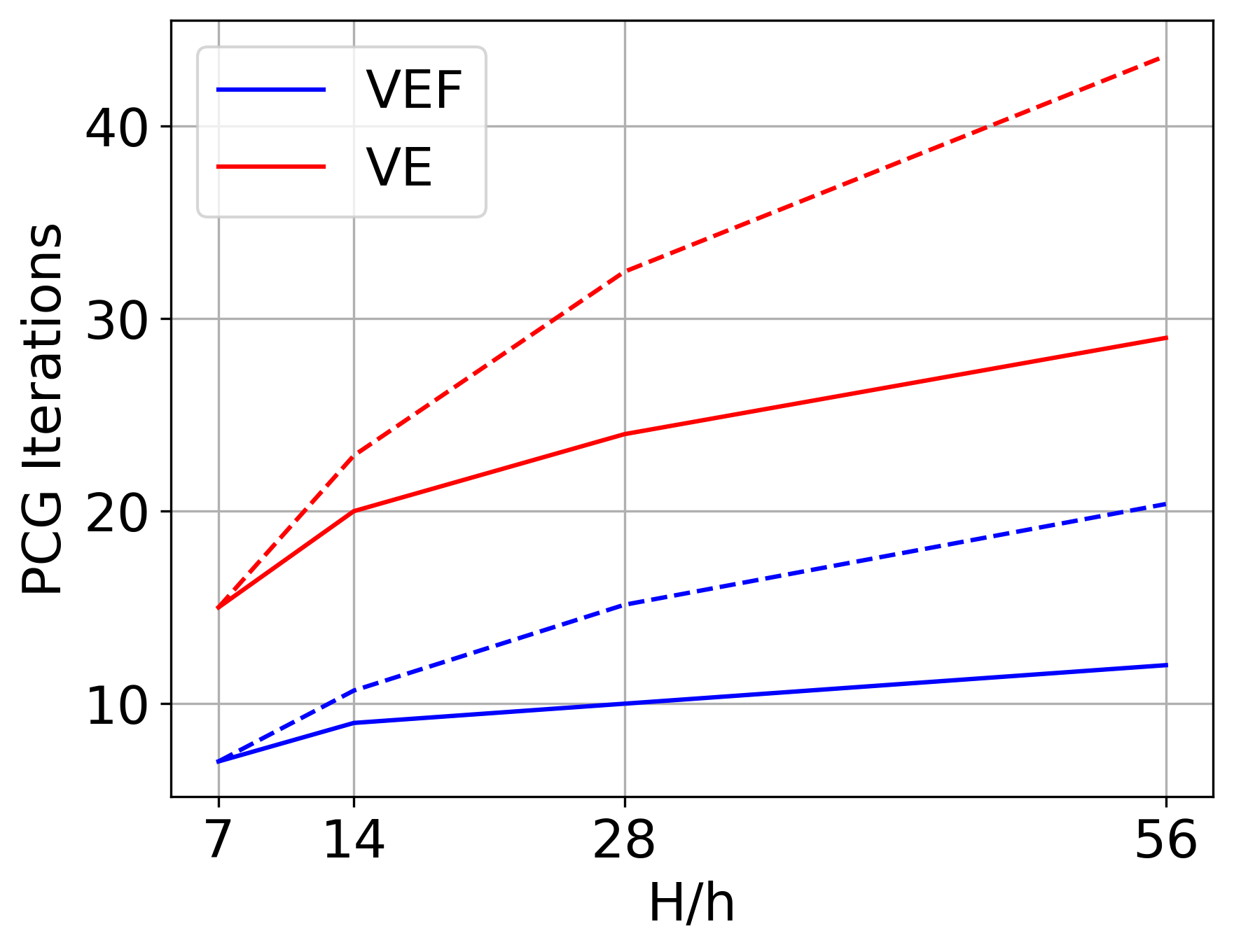}
    \end{subfigure}
    \begin{subfigure}[b]{0.4\textwidth}
        \includegraphics[width=\textwidth]{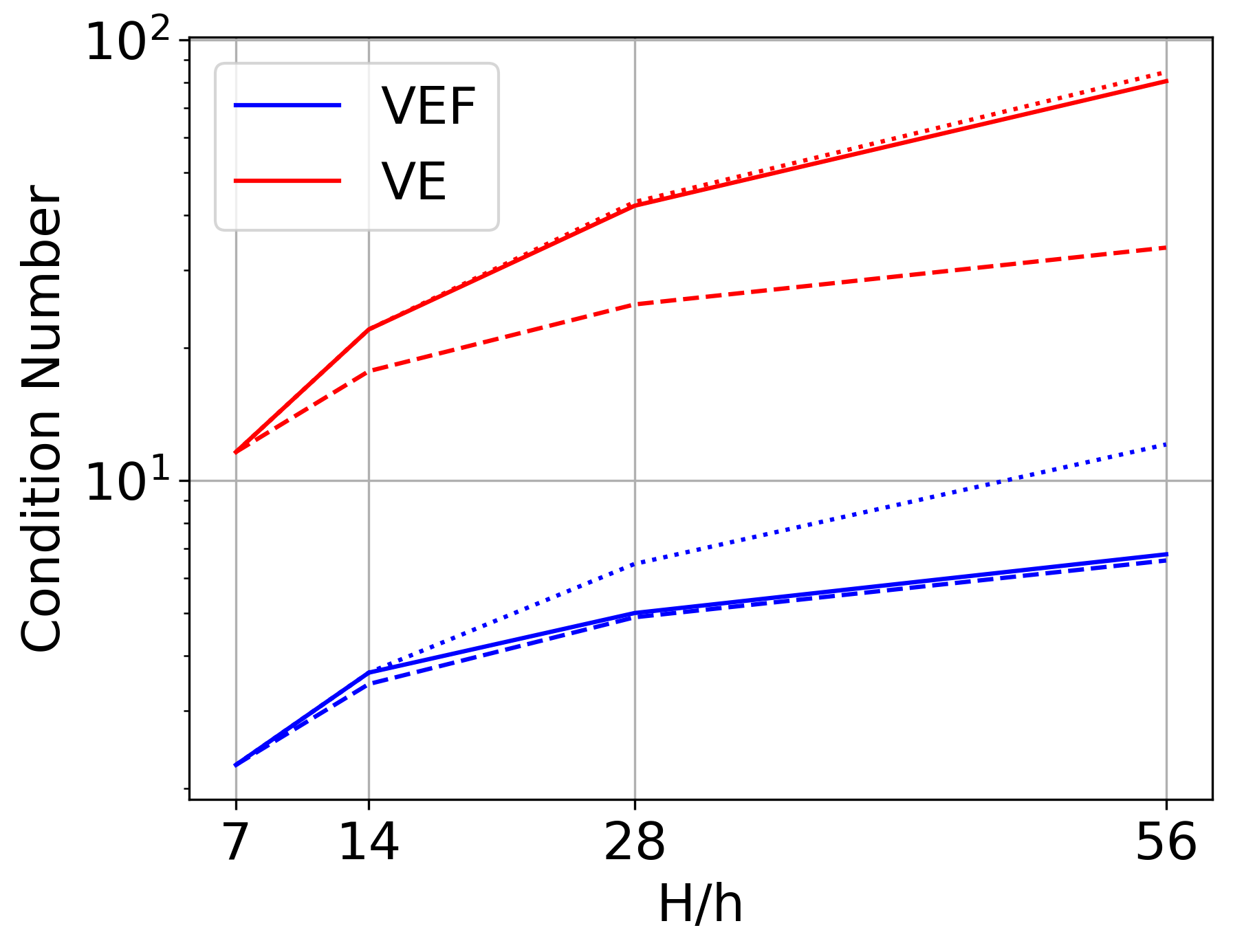}
    \end{subfigure}
    \caption{Iterations needed to converge to a relative residual tolerance of $10^{-6}$ (left) and condition number estimates (right) for an increasing refinement level. Solid lines represent the obtained numerical results. Refining the problem increases $\frac{H}{h}$, resulting in poly-logarithmic increase in the condition number and in the iteration count (dashed lines). For the coarse space containing only vertex and edge constraints, the actual growth of the condition number appears to be closer to linear (dotted line). Here, the dashed line is the graph of $\bigg(1 + \log\frac{H}{h}\bigg)^2$ scaled such that it intersects with the first measured data point and the dotted line is a linear interpolation of the first two measured data points.}
    \label{fig:refinement}
\end{figure}

\begin{figure}
    \centering
    \begin{subfigure}[b]{0.4\textwidth}
        \includegraphics[width=\textwidth]{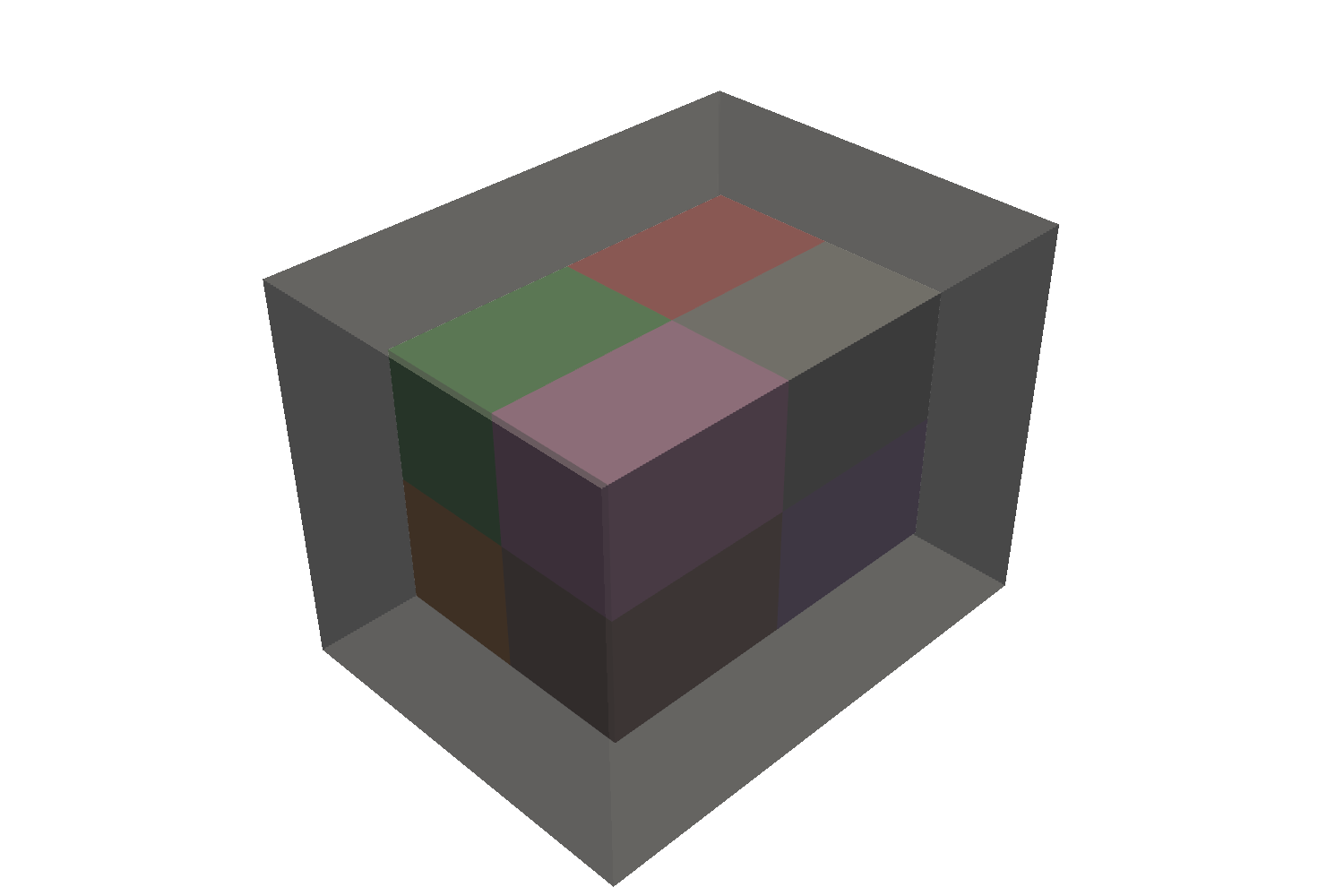}
    \end{subfigure}
    \begin{subfigure}[b]{0.4\textwidth}
        \includegraphics[width=\textwidth]{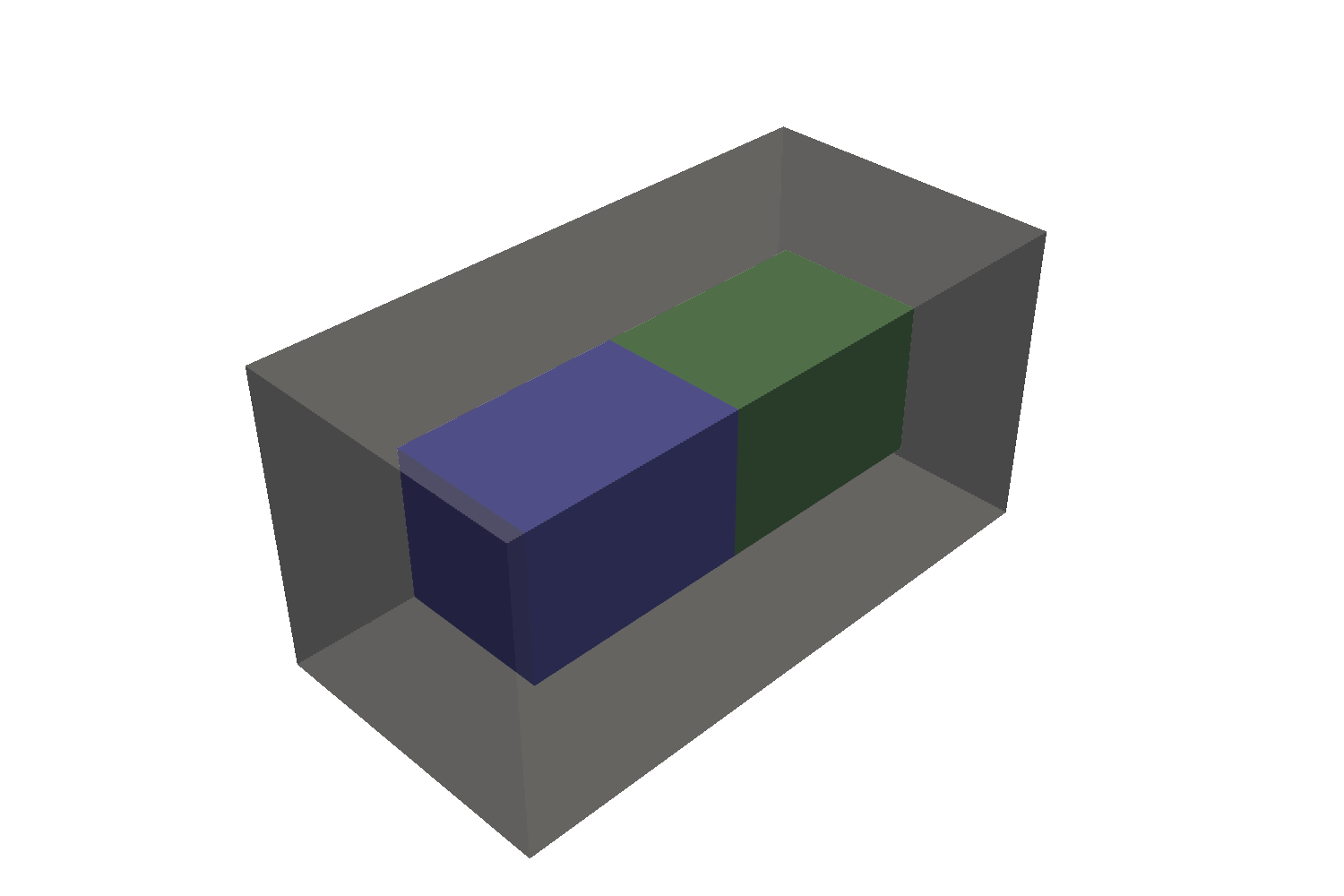}
    \end{subfigure}
    \caption{Simplified meshes containing 8 (left) and 2 (right) convex myocytes floating in extracellular space.}
    \label{fig:simple-meshes}
\end{figure}

\begin{figure}
    \centering
    \begin{subfigure}[b]{0.4\textwidth}
        \includegraphics[width=\textwidth]{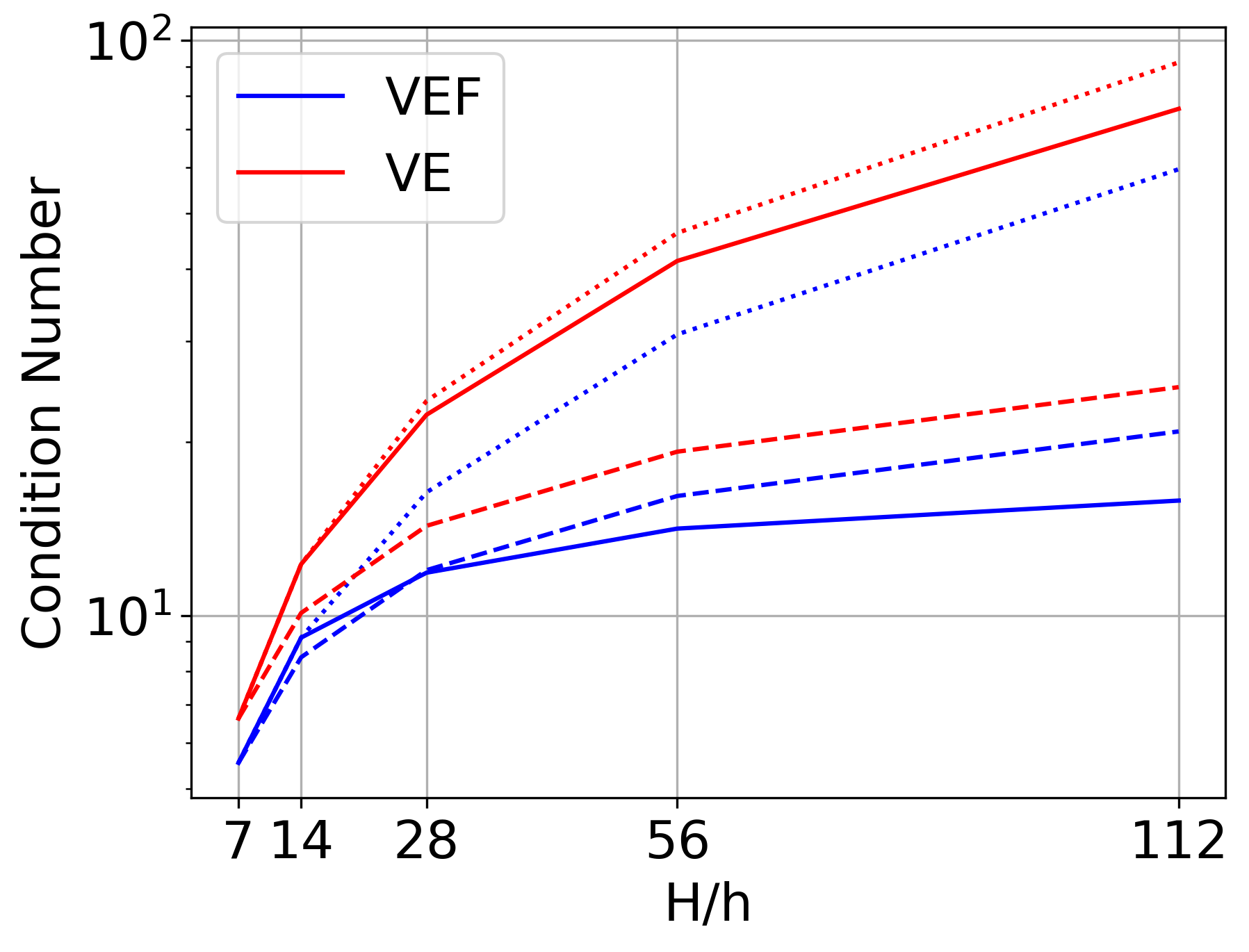}
    \end{subfigure}
    \begin{subfigure}[b]{0.4\textwidth}
        \includegraphics[width=\textwidth]{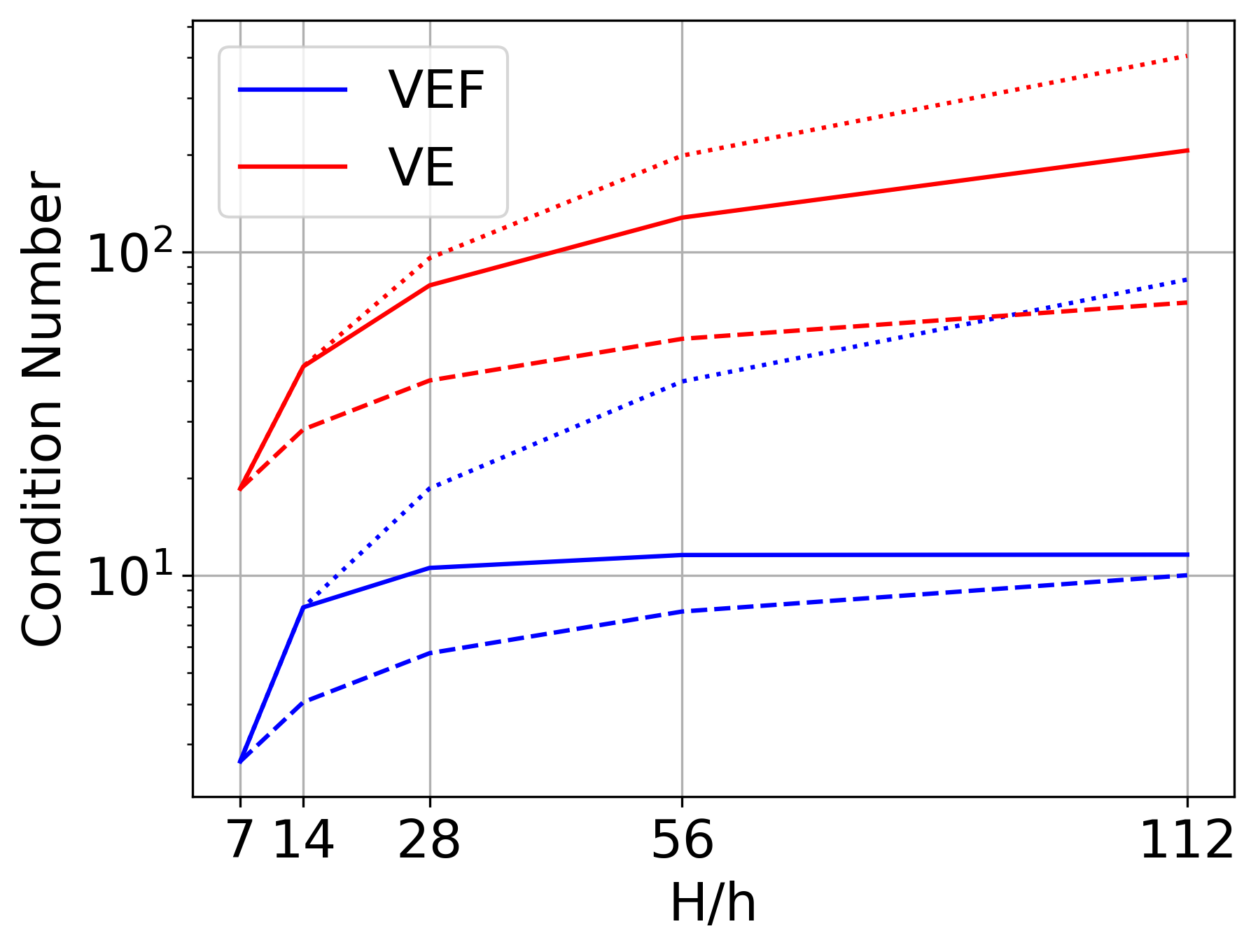}
    \end{subfigure}
    \caption{Condition number estimates for increasing refinement levels of simple geometries as shown in \Cref{fig:simple-meshes}. Solid lines represent the obtained numerical results. For the mesh made up from 8 cells (left), we already see a stronger sub-linear trend than for our repetitive test geometry while for the simplest case containing only two convex cells (right), the behavior shows a poly-logarithmic trend (dashed lines), aligning with the theoretical results. For the coarse space containing only vertex and edge constraints, the actual growth of the condition number appears to be closer to linear (dotted line). We note that for the simplest geometry, the condition number estimates are generally higher than for the more complex geometries, which is most likely due to the complete lack of primal (vertex) constraints.}
    \label{fig:refinement-simple}
\end{figure}

\section{Conclusion}

In this paper, we have constructed and analyzed the Balancing Domain Decomposition by Constraints (BDDC) preconditioner for composite DG-type discretizations of cardiac cell-by-cell models in three spatial dimensions. We have derived theoretical convergence results and used an experimental evaluation to confirm these bounds for a synthetic test case with a repetitive geometry. We also demonstrated the scalability and quasi-optimality of the preconditioner and the convergence being independent of conductivity coefficients. 
Since we mostly focused on the theoretical analysis and numerical convergence, we have not included very detailed profiling results, which could be interesting for a potential follow-up on a multilevel or approximate BDDC study.
Future research will include the integration of the method into large-scale cardiac simulations, moving further away from the regularly shaped geometries considered by the theory and towards more organically, irregularly shaped myocytes as they would appear in the human heart.
\section{Acknowledgements}
This work was supported by the European High-Performance Computing Joint Undertaking EuroHPC under grant agreement No 955495 (MICROCARD) co-funded by the Horizon 2020 programme of the European Union (EU), the French National Research Agency ANR, the German Federal Ministry of Education and Research, the Italian ministry of economic development, the Swiss State Secretariat for Education, Research and Innovation, the Austrian Research Promotion Agency FFG, and the Research Council of Norway. \newline
This work was supported by the MICROCARD-2 project (project ID 101172576).
The project is supported by the EuroHPC Joint Undertaking and its
members (including top-up funding by ANR, BMBF, and Ministero dello
sviluppo economico).
Funded by the European Union. Views and opinions expressed are
however those of the author(s) only and do not necessarily reflect
those of the European Union or EuroHPC.  Neither the European Union
nor EuroHPC can be held responsible for them. \newline
The authors acknowledge the EuroHPC award for the availability of high-performance computing resources on Karolina (EU2024D09-042).
N.M.M.H., L.F.P. and S.S. have been supported by grants of Istituto Nazionale di Alta Matematica (INDAM-GNCS) and by MUR (PRIN 202232A8AN\_002, PRIN 202232A8AN\_003, PRIN P2022B38NR\_001 and PRIN P2022B38NR\_002) funded by European Union - Next Generation EU.

\bibliography{main}
\end{document}